\documentclass[final,1p,times]{elsarticle}%
\usepackage{amsthm}
\usepackage{amssymb}
\usepackage{amsmath}
\usepackage{amsfonts}
\usepackage{graphicx}%
\providecommand{\U}[1]{\protect\rule{.1in}{.1in}}
\newtheorem{proposition}{Proposition}
\newtheorem{theorem}{Theorem}
\newtheorem{corollary}{Corollary}
\newtheorem{lemma}{Lemma}
\newtheorem{remark}{Remark}
\journal{}
\begin{document}
\begin{frontmatter}
\title{Extension of the G\"{u}nter derivatives to Lipschitz domains and application
to the boundary potentials of elastic waves}
\author[INSA]{A.~Bendali\corref{cor_author}}
\ead{abendali@insa-toulouse.fr}
\author[UPPA,M3D]{S.~Tordeux}
\ead{sebastien.tordeux@inria.fr}
\cortext[cor_author]{Corresponding author}
\address[INSA]{Universit\'e de Toulouse, Institut Math\'ematique de Toulouse,\\
D\'epartement de G\'enie Math\'ematique, INSA, Toulouse, France}
\address[UPPA]{University of Pau (France)}
\address[M3D]{INRIA Bordeaux Sud-Ouest Magique-$3$D Team, Pau (France)}
\begin{abstract}
The scalar G\"{u}nter derivatives of a function defined on
the boundary of a three-dimensional domain are expressed as components (or their opposites)
of the tangential vector rotational of this function in the canonical orthonormal
basis of the ambient space. This in particular implies that these derivatives
define bounded operators from $H^s$ into $H^{s-1}$ for $0 \le s \le 1$ on the boundary
of a Lipschitz domain, and can easily be implemented in boundary element codes.
Regularization techniques
for the trace and the traction of elastic waves potentials, previously built
for a domain of class $\mathcal{C}^2$, can thus be extended to the Lipschitz case.
In particular, this yields an elementary way  to establish
the mapping properties of elastic wave potentials
from those of the Helmholtz equation without resorting to the more
advanced theory for elliptic systems. Some attention is finally paid to the two-dimensional
case.
\end{abstract}
\end{frontmatter}

\section{Introduction}

All along this paper, $\Omega^{+}$ and $\Omega^{-}=\mathbb{R}^{3}%
\smallsetminus\overline{\Omega^{+}}$ respectively designate a bounded
Lipschitz domain of $\mathbb{R}^{3}$, and its exterior. As a result,
$\Omega^{+}$ and $\Omega^{-}$ share a common boundary denoted by
$\partial\Omega$. It is well-known that $\partial\Omega$ is endowed with a
Lebesgue surface measure $s$, and that it has an unit normal $\boldsymbol{n}$
(see figure~\ref{fig_dom}), defined $s$-almost everywhere, pointing outward
from $\Omega^{+}$ (cf., for example, \cite[p. 96]{McLean:00}). Vectors with
three components $a_{j}$ $\left(  j=1,2,3\right)  $, either real or complex,
are identified to column-vectors
\[
\boldsymbol{a}=\left[
\begin{array}
[c]{c}%
a_{1}\\
a_{2}\\
a_{3}%
\end{array}
\right]  .
\]
The bilinear form underlying the scalar product of two such vectors
$\boldsymbol{a}$ and $\boldsymbol{b}$ is given by
\[
\boldsymbol{a\cdot b}=\boldsymbol{a}^{\top}\boldsymbol{b}=\boldsymbol{b}%
^{\top}\boldsymbol{a=}\sum_{j=1}^{3}a_{j}b_{j}%
\]
where $\boldsymbol{a}^{\top}$ is the transpose of $\boldsymbol{a}$.%

\begin{figure}[pbh]%
\centering
\includegraphics[
height=6.1352cm,
width=5.6739cm
]%
{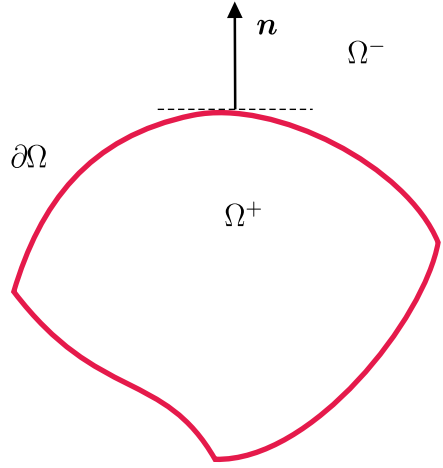}%
\caption{Schematic view of the geometry}%
\label{fig_dom}%
\end{figure}

Usual notation in the theory of Partial Differential Equations
\cite{Taylor:96} will be used without further comment. We just mention that we
make use of the following Fr\'{e}chet spaces, defined for any integer $m\geq0$
by%
\[
H_{\text{loc}}^{m}\left(  \mathbb{R}^{3}\right)  =\left\{  v\in\mathcal{D}%
^{\prime}\left(  \mathbb{R}^{3}\right)  ;\;\varphi v\in H^{m}\left(
\mathbb{R}^{3}\right)  ,\;\forall\varphi\in\mathcal{D}\left(  \mathbb{R}%
^{3}\right)  \right\}
\]

\[
H_{\text{loc}}^{m}\left(  \overline{\Omega^{-}}\right)  :=\left\{
v\in\mathcal{D}^{\prime}\left(  \Omega^{-}\right)  ;\;\exists V\in
H_{\text{loc}}^{m}\left(  \mathbb{R}^{3}\right)  ,\;v=V|_{\Omega^{-}}\right\}
,
\]
and by
\[
H_{\text{comp}}^{m}\left(  \mathbb{R}^{3}\right)  =\left\{  v\in H^{m}\left(
\mathbb{R}^{3}\right)  ;\;\exists R>0,\;v|_{\left\vert x\right\vert \geq
R}=0\right\}
\]%
\[
H_{\text{comp}}^{m}\left(  \overline{\Omega^{-}}\right)  =\left\{  v\in
H^{m}\left(  \Omega^{-}\right)  ;\;\exists V\in H_{\text{comp}}^{m}\left(
\mathbb{R}^{3}\right)  ,\;v=V|_{\Omega^{-}}\right\}  .
\]
With similar definitions, it is trivially true that $H_{\text{loc}}^{m}\left(
\overline{\Omega^{+}}\right)  =H_{\text{comp}}^{m}\left(  \overline{\Omega
^{+}}\right)  =H^{m}\left(  \Omega^{+}\right)  $. Below, we conveniently use
the unified notation $H_{\text{loc}}^{m}\left(  \overline{\Omega^{\pm}%
}\right)  $ and $H_{\text{comp}}^{m}\left(  \overline{\Omega^{\pm}}\right)  $
to refer to both of these spaces. Instead of $H^{0}$, we use the more
conventional notation $L^{2}$.

We denote by $u^{+}=\left(  u|_{\Omega^{+}}\right)  |_{\partial\Omega}$ (resp.
$u^{-}=\left(  u|_{\Omega^{-}}\right)  |_{\partial\Omega}$) the trace of $u$
on $\partial\Omega$ from the values $u|_{\Omega^{+}}$ of $u$ in $\Omega^{+}$
(resp. $u|_{\Omega^{-}}$ in $\Omega^{-}$). For simplicity, we omit to
explicitly mention the trace when the related function has zero jump across
$\partial\Omega$. We also adopt a classical way to denote functional spaces of
vector fields having their components in some scalar functional space. For
example, $H^{s}\left(  \partial\Omega;\mathbb{C}^{3}\right)  $ stands for the
space of vector fields $\boldsymbol{u}$ whose components $u_{j}$ $\left(
j=1,2,3\right)  $ are in $H^{s}\left(  \partial\Omega\right)  $.

For $1\leq i,j\leq3$ and $u\in H_{\text{loc}}^{2}\left(  \mathbb{R}%
^{3}\right)  $, the G\"{u}nter derivative
\begin{equation}
\mathcal{M}_{ij}^{\left(  \boldsymbol{n}\right)  }u=n_{j}\partial_{x_{i}%
}u-n_{i}\partial_{x_{j}}u \label{Mij}%
\end{equation}
is well-defined as a function in $L^{2}\left(  \partial\Omega\right)  $ since
the traces of $\partial_{x_{i}}u$ and $\partial_{x_{j}}u$ are in
$H^{1/2}\left(  \partial\Omega\right)  $ and the components $n_{i}$ and
$n_{j}$ of the normal $\boldsymbol{n}$ to $\partial\Omega$ are in $L^{\infty
}\left(  \partial\Omega\right)  $. It is worth recalling that if $\Omega^{+}$
is a bit more regular, say a $\mathcal{C}^{1,1}$-domain for example (cf.,
\cite[p. 90]{McLean:00} for the definition of a $\mathcal{C}^{k}$-domain
(resp. $\mathcal{C}^{k,\alpha}$-domain), also referred to as a domain of class
$\mathcal{C}^{k}$ (resp. $\mathcal{C}^{k,\alpha}$)), $\mathcal{M}%
_{ij}^{\left(  \boldsymbol{n}\right)  }u$ is in $H^{1/2}\left(  \partial
\Omega\right)  $. Seemingly, there is a loss of one-half order of regularity
when considering a domain which is only Lipschitz. The purpose of this paper
is precisely to show that this one-half order of regularity can be restored
for functions in lower order Sobolev spaces.

Let us first recall some well-established properties of the G\"{u}nter
derivatives when $\Omega^{+}$ is at least a $\mathcal{C}^{1,1}$-domain. Let
\[
\left\{  \boldsymbol{e}_{j}=\left[  \delta_{1j},\delta_{2j},\delta
_{3j}\right]  ^{\top}\right\}  _{j=1}^{3}\text{ (}\delta_{ij}=1\text{ if
}i=j\text{ and }0\text{ otherwise)}%
\]
be the canonical basis of $\mathbb{R}^{3}$ so that $n_{j}=\boldsymbol{n}%
\cdot\boldsymbol{e}_{j}$ for $j=1,2,3$. Define for $1\leq i\neq j\leq3$%
\begin{equation}
\boldsymbol{\tau}_{ij}=n_{j}\boldsymbol{e}_{i}-n_{i}\boldsymbol{e}_{j}.
\label{Tau_ij}%
\end{equation}
Clearly%
\[
\mathcal{M}_{ij}^{\left(  \boldsymbol{n}\right)  }u=\boldsymbol{\nabla}%
u\cdot\boldsymbol{\tau}_{ij}=\partial_{\boldsymbol{\tau}_{ij}}u
\]
with
\begin{equation}
\boldsymbol{\tau}_{ij}\cdot\boldsymbol{n}=0. \label{Tau_ij.n}%
\end{equation}
As a result, $\mathcal{M}_{ij}^{\left(  \boldsymbol{n}\right)  }$ is a
tangential derivative on $\partial\Omega$, meaning in particular, at least for
$u\in\mathcal{C}^{1}\left(  \mathbb{R}^{3}\right)  $ and $\Omega^{+}$ a
$\mathcal{C}^{1}$-domain, that $\mathcal{M}_{ij}^{\left(  \boldsymbol{n}%
\right)  }u$ can be calculated without resorting to interior values of $u$ in
$\Omega^{+}$ or in $\Omega^{-}$.

These operators were introduced by G\"{u}nter \cite{Gunter:53}. It was
discovered later \cite{Kupradze:79} that they can be used for bringing out
important relations linking the boundary layer potentials of the Lam\'{e}
system to those of the Laplace equation (see \cite[p. 314]{Kupradze:79} and
\cite[p. 48]{Hsiao-Wendland:08}). They were then employed to more conveniently
express the traction of the double layer elastic potential (see \cite{Han:94}
and \cite[p. 49]{Hsiao-Wendland:08}). These approaches were recently extended
to the elastic wave boundary layer potentials by Le\ Lou\"{e}r
\cite{LeLouer:14,Lelouerv1:12}. All these results were derived under the
assumption that $\Omega^{+}$ is a $\mathcal{C}^{2}$-domain (actually,
$\mathcal{C}^{1,1}$- is enough). It is the aim of this paper, by defining the
G\"{u}nter derivatives for a Lipschitz domain, that is, a $\mathcal{C}^{0,1}%
$-domain, to similarly handle geometries more usual in the applications. More
importantly, it is possible in this way to deal with boundary element
approximations of the traction of single- and double-layer potentials of
Lam\'{e} static elasticity and elastic wave systems almost as easily as for
the Laplace or the Helmholtz equation.

Actually, in connection with elasticity potential layers, G\"{u}nter
derivatives are involved as entries $\mathcal{M}_{ij}^{\left(  \boldsymbol{n}%
\right)  }$ $\left(  i,j=1,2,3\right)  $ of the skew-symmetric matrix
$\mathcal{M}^{\left(  \boldsymbol{n}\right)  }$ acting on vector-valued
functions $\boldsymbol{u}$%
\[
\left(  \mathcal{M}^{\left(  \boldsymbol{n}\right)  }\boldsymbol{u}\right)
_{i}=\sum_{j=1}^{3}\mathcal{M}_{ij}^{\left(  \boldsymbol{n}\right)  }%
u_{j}\quad\left(  i=1,2,3\right)  ,
\]
$\left(  \mathcal{M}^{\left(  \boldsymbol{n}\right)  }\boldsymbol{u}\right)
_{i}$ and $u_{j}$ $\left(  j=1,2,3\right)  $ being the respective components
of $\mathcal{M}^{\left(  \boldsymbol{n}\right)  }\boldsymbol{u}$ and
$\boldsymbol{u}$. In \cite{Hsiao-Wendland:08}, this matrix is called the
G\"{u}nter derivatives in matrix form. We find it more convenient to refer to
$\mathcal{M}^{\left(  \boldsymbol{n}\right)  }\boldsymbol{u}$ as the
G\"{u}nter derivative matrix.

The G\"{u}nter derivative matrix actually give rise to a multi-faceted
operator, with various expressions, which led to real progresses in the
context of Lam\'{e} static elasticity boundary layer potentials
\cite{Kupradze:79,Hsiao-Wendland:08,LeLouer:14,Han:94} or in the design of
preconditioning techniques for the boundary integral formulations in the
scattering of elastic waves \cite{Darbas-Lelouer:15}. Other ways to write
$\mathcal{M}^{\left(  \boldsymbol{n}\right)  }$\ do not seem to have been
connected with the G\"{u}nter derivatives \cite{Costabel:91,Costabel-Dauge:96}%
. However, all these expressions require either interior values, as for
example for above direct definition (\ref{Mij}) of $\mathcal{M}_{ij}^{\left(
\boldsymbol{n}\right)  }$, or curvature terms of $\partial\Omega$ as recalled
below, making problematic their effective implementation in boundary element
codes or in a preconditioning technique. It is among the objectives of this
paper to address this issue.

The outline of the paper is as follows. In section \ref{section1}, we first
show that $\mathcal{M}_{ij}^{\left(  \boldsymbol{n}\right)  }u$ corresponds to
a component (or its opposite) of the tangential vector rotational
$\boldsymbol{\nabla}_{\partial\Omega}u\times\boldsymbol{n}$ of $u$ in an
orthonormal basis of the ambient space. This feature, in addition to some
duality properties, enable us to define this derivative as a bounded operator
from $H^{s}\left(  \partial\Omega\right)  $ into $H^{s-1}\left(
\partial\Omega\right)  $ for $0\leq s\leq1$. This is actually equivalent to
arguing that $u\rightarrow\boldsymbol{\nabla}_{\partial\Omega}u\times
\boldsymbol{n}$ is a bounded operator from $H^{s}\left(  \partial
\Omega\right)  $ into $H^{s-1}\left(  \partial\Omega;\mathbb{C}^{3}\right)  $
for $0\leq s\leq1$, a result which was established for $s=1/2$ in
\cite{Buffa-et-al:02} from a different technique. It then follows that its
transpose, yielding the surface rotational $\boldsymbol{\nabla}_{\partial
\Omega}\cdot\boldsymbol{u}\times\boldsymbol{n}$ of a vector field
$\boldsymbol{u}$ \cite[p. 73]{Nedelec:01},\ defines also a bounded operator
from $H^{s}\left(  \partial\Omega;\mathbb{C}^{3}\right)  $ into $H^{s-1}%
\left(  \partial\Omega\right)  $ for $0\leq s\leq1$. It is worth noting that
in \cite[p. 73]{Nedelec:01} the surface rotational was considered for
tangential vector fields only. However, the cross-product involved in the
expression of this operator makes it possible to extend this definition to a
general vector field. We show then that G\"{u}nter derivatives $\mathcal{M}%
_{ij}^{\left(  \boldsymbol{n}\right)  }$ can be expressed as differential
forms to retrieve an integration by parts formula relatively to these
operators on a patch of $\partial\Omega$. Even if this formula was already
established by direct calculation in \cite{Gunter:53}, we think that the
formalism of differential forms is more appropriate for understanding the
basic principle underlying its derivation. It is used here to get explicit
expressions for the G\"{u}nter derivatives of a piecewise smooth function
defined on a the boundary of a curved polyhedron. This way to write these
derivatives is fundamental in the effective implementations of boundary
element codes. In section \ref{section2}, we begin with some recalls on other
previous expressions for G\"{u}nter derivative matrix $\mathcal{M}^{\left(
\boldsymbol{n}\right)  }$. With the help of a vectorial Green formula, partly
introduced in \cite{Knauff-Kress:79} and in a more complete form in
\cite{Costabel:91,Costabel-Dauge:96}, we derive a useful volume variational
expression for $\mathcal{M}^{\left(  \boldsymbol{n}\right)  }$. As an
application in section \ref{section3}, we extend the regularization techniques
(the way for expressing non-integrable kernels involved in boundary layer
potentials in terms of integrals converging in the usual meaning) devised by
Le\ Lou\"{e}r \cite{LeLouer:14,Lelouerv1:12} for the elastic wave layer
potentials to Lipschitz domains. It is worth recalling that due to its
importance in practical implementations of numerical solvers for elastic wave
scattering problems, several other regularizations techniques, much more
involved in our opinion, have been already proposed (cf., for example,
\cite{Nedelec:86,Nishimura-kobayashi:89,Becache-et-al:93,Liu-Rizzo:93,Grany-Chang:92,Rizzo-et-al:85}
to cite a few). Finally, in Section~\ref{section4}, making use of the
connection between two- and three-dimensional Green kernels for the Helmholtz
equation, we transpose the regularization techniques in the spatial scale to
planar elastic waves.

\section{Extension of the G\"{u}nter derivatives to a Lipschitz
domain\label{section1}}

We first establish some mapping properties of the G\"{u}nter derivatives in
the framework of a Lipschitz domain. We next show that they can be written as
differential 2-forms, up to a Hodge star identification. This will allow us to
retrieve an integration by parts formula on the patches of $\partial\Omega$.
That yields an expression of these derivatives well suited for boundary
element codes.

\subsection{Mapping properties of the G\"{u}nter derivatives for Lipschitz
domains}

Property (\ref{Tau_ij.n}) ensures that $\mathcal{M}_{ij}^{\left(
\boldsymbol{n}\right)  }$ is a first-order differential operator tangential to
$\partial\Omega$ in the sense of \cite[p. 147]{McLean:00}. This immediately
leads to the following first mapping property whose proof is given in
Lemma\ 4.23 of this reference.

\begin{proposition}
There exists a constant $C$ independent of $u\in H^{1}\left(  \partial
\Omega\right)  $ such that%
\begin{equation}
\left\Vert \mathcal{M}_{ij}^{\left(  \boldsymbol{n}\right)  }u\right\Vert
_{L^{2}\left(  \partial\Omega\right)  }\leq C\left\Vert u\right\Vert
_{H^{1}\left(  \partial\Omega\right)  }. \label{Estimate_H1}%
\end{equation}

\end{proposition}

To go further, we make the following observation which, surprisingly enough,
does not seem to have been done before. \ It consists in noting that vector
$\boldsymbol{\tau}_{ij}$, defined in (\ref{Tau_ij}), can be written under the
following form using the elementary double product formula%
\[
\boldsymbol{\tau}_{ij}=\left(  \boldsymbol{n}\cdot\boldsymbol{e}_{j}\right)
\boldsymbol{e}_{i}-\left(  \boldsymbol{n}\cdot\boldsymbol{e}_{i}\right)
\boldsymbol{e}_{j}=\boldsymbol{n}\times\left(  \boldsymbol{e}_{i}%
\times\boldsymbol{e}_{j}\right)  \;\left(  1\leq i,j\leq3\right)  .
\]
In this way, using the properties of the mixed product, we can also put
G\"{u}nter derivative $\mathcal{M}_{ij}^{\left(  \boldsymbol{n}\right)  }u$ in
the following form%
\begin{equation}
\mathcal{M}_{ij}^{\left(  \boldsymbol{n}\right)  }u=\boldsymbol{\nabla}%
u\cdot\boldsymbol{n}\times\left(  \boldsymbol{e}_{i}\times\boldsymbol{e}%
_{j}\right)  =\boldsymbol{\nabla}u\times\boldsymbol{n}\cdot\boldsymbol{e}%
_{i}\times\boldsymbol{e}_{j}. \label{Guxneiej}%
\end{equation}

Indeed, formula (\ref{Guxneiej}) expresses $\mathcal{M}_{ij}^{\left(
\boldsymbol{n}\right)  }u$ as a component (or its opposite) of the tangential
vector rotational of $u$ in the canonical basis of $\mathbb{R}^{3}$
\[
\mathcal{M}_{ij}^{\left(  \boldsymbol{n}\right)  }u=\boldsymbol{\nabla
}_{\partial\Omega}u\times\boldsymbol{n}\cdot\boldsymbol{e}_{i}\times
\boldsymbol{e}_{j}.
\]
(See \cite[p. 69]{Nedelec:01} for the definition and properties of the
tangential gradient $\boldsymbol{\nabla}_{\partial\Omega}u$ and the tangential
vector rotational of a function when, for example, $\Omega^{+}$ is a
$\mathcal{C}^{2}$-domain.)

We have next the following lemma which is established in a less
straightforward way in \cite{Gunter:53} for a $\mathcal{C}^{1,\alpha}$-domain
$\left(  0<\alpha\leq1\right)  $.

\begin{lemma}
For $u$ and $v$ in $\mathcal{C}_{\text{comp}}^{1}\left(  \mathbb{R}%
^{3}\right)  $, the following integration by parts formula%
\begin{equation}
\int_{\partial\Omega}v\mathcal{M}_{ij}^{\left(  \boldsymbol{n}\right)
}u\ ds=\int_{\partial\Omega}u\mathcal{M}_{ji}^{\left(  \boldsymbol{n}\right)
}v\ ds \label{Int_Sym}%
\end{equation}
holds true.
\end{lemma}

\begin{proof}
The proof directly follows from the following simple observation%
\[
\boldsymbol{\nabla}\times\left(  uv\boldsymbol{e}_{i}\times\boldsymbol{e}%
_{j}\right)  =v\boldsymbol{\nabla}\times\left(  u\boldsymbol{e}_{i}%
\times\boldsymbol{e}_{j}\right)  -u\boldsymbol{\nabla}\times\left(
v\boldsymbol{e}_{j}\times\boldsymbol{e}_{i}\right)
\]
and Green's formula in Lipschitz domains \cite[Th. 3.34]{McLean:00}%
\begin{align*}
\int_{\Omega^{\pm}}\underset{=0}{\underbrace{\boldsymbol{\nabla}%
\cdot\boldsymbol{\nabla}\times\left(  uv\boldsymbol{e}_{i}\times
\boldsymbol{e}_{j}\right)  }}\ dx  & =\pm\int_{\partial\Omega}\left(
\boldsymbol{\nabla}\times\left(  uv\boldsymbol{e}_{i}\times\boldsymbol{e}%
_{j}\right)  \right)  \cdot\boldsymbol{n}ds\\
& =\mp\left(  \int_{\partial\Omega}v\mathcal{M}_{ij}^{\left(  \boldsymbol{n}%
\right)  }u\ ds-\int_{\partial\Omega}u\mathcal{M}_{ji}^{\left(  \boldsymbol{n}%
\right)  }v\ ds\right)  .
\end{align*}

\end{proof}

We then come to the following theorem embodying optimal mapping properties of
the G\"{u}nter derivatives.

\begin{theorem}
\label{Gunter_map}Under the above assumption that $\Omega^{+}$ is a bounded
Lipschitz domain, G\"{u}nter derivative $\mathcal{M}_{ij}^{\left(
\boldsymbol{n}\right)  }$ can be extended in a bounded linear operator from
$H^{s}\left(  \partial\Omega\right)  $ into $H^{s-1}\left(  \partial
\Omega\right)  $ for $0\leq s\leq1.$
\end{theorem}

\begin{proof}
It is a straightforward consequence of estimate (\ref{Estimate_H1}) and
symmetry property (\ref{Int_Sym}) by duality and interpolation techniques.
\end{proof}

\begin{corollary}
\label{tang_vect_rot}Under the general assumptions of the above theorem, the
tangential vector rotational defines a bounded linear operator $u\rightarrow
\boldsymbol{\nabla}_{\partial\Omega}u\times\boldsymbol{n}$ from $H^{s}\left(
\partial\Omega\right)  $ into $H^{s-1}\left(  \partial\Omega;\mathbb{C}%
^{3}\right)  $ for $0\leq s\leq1$. Consequently, the surface rotational gives
rise to a bounded operator $\boldsymbol{u}\in H^{s}\left(  \partial
\Omega;\mathbb{C}^{3}\right)  \rightarrow\boldsymbol{\nabla}_{\partial\Omega
}\cdot\boldsymbol{u}\times\boldsymbol{n}\in$\ $H^{s-1}\left(  \partial
\Omega\right)  $ for $0\leq s\leq1$.
\end{corollary}

\begin{proof}
Immediate since the components of $\boldsymbol{\nabla}_{\partial\Omega}%
u\times\boldsymbol{n}$ are nothing else but G\"{u}nter derivatives and the
surface rotational is the transpose of the tangential vector rotational.
\end{proof}

\begin{remark}
When $u$ and $v$ are the respective traces of functions in $H^{1}\left(
\mathbb{R}^{3}\right)  $, it is established in \cite[p. 855]{Buffa-et-al:02}
that $\boldsymbol{\nabla}u\times\boldsymbol{n}$ is well-defined in $H_{\Vert
}^{-1/2}\left(  \partial\Omega;\mathbb{C}^{3}\right)  $, the dual space of
\[
H_{\Vert}^{1/2}\left(  \partial\Omega;\mathbb{C}^{3}\right)  =\left\{
\boldsymbol{v}\in L^{2}\left(  \partial\Omega;\mathbb{C}^{3}\right)
;\;\boldsymbol{v}=\boldsymbol{n}\times(\boldsymbol{w}\times\boldsymbol{n}%
),\;\boldsymbol{w}\in H^{1}\left(  \Omega^{\pm};\mathbb{C}^{3}\right)
\right\}
\]
equipped with the graph norm and that $\boldsymbol{\nabla}u\times
\boldsymbol{n}$ depends on the trace $u|_{\partial\Omega}$ of $u$ on
$\partial\Omega$ \cite[p. 855]{Buffa-et-al:02} only. It is also proved in this
paper \cite[formulae (15) p. 850 and Lemma 2.3 p. 851]{Buffa-et-al:02}\ that
$H_{\Vert}^{-1/2}\left(  \partial\Omega;\mathbb{C}^{3}\right)  $ can be
identified to a closed subspace of $H^{-1/2}\left(  \partial\Omega
;\mathbb{C}^{3}\right)  $. This is the particular case corresponding to
$s=1/2$ which has been previously mentioned.
\end{remark}

The following symmetry result is known for a long time in the case of smoother
domains and more regular functions \cite[p. 284]{Kupradze:79} and is an
immediate consequence of the definition of the G\"{u}nter derivative matrix
and the symmetry and mapping properties of $\mathcal{M}_{ij}^{\left(
\boldsymbol{n}\right)  }$.

\begin{corollary}
G\"{u}nter derivative matrix $\mathcal{M}^{\left(  \boldsymbol{n}\right)  }$
defines a bounded linear operator from $H^{s}\left(  \partial\Omega
;\mathbb{C}^{3}\right)  $ into $H^{s-1}\left(  \partial\Omega;\mathbb{C}%
^{3}\right)  $ for $0\leq s\leq1$ with the following symmetry property%
\begin{equation}
\left\langle \boldsymbol{v},\mathcal{M}^{\left(  \boldsymbol{n}\right)
}\boldsymbol{u}\right\rangle _{1-s,\partial\Omega}=\left\langle \boldsymbol{u}%
,\mathcal{M}^{\left(  \boldsymbol{n}\right)  }\boldsymbol{v}\right\rangle
_{s,\partial\Omega},\;\boldsymbol{u}\in H^{s}\left(  \partial\Omega
;\mathbb{C}^{3}\right)  ,\boldsymbol{v}\in H^{1-s}\left(  \partial
\Omega;\mathbb{C}^{3}\right)  . \label{Sym_M}%
\end{equation}
For simplicity, we keep the same notation for the bilinear form underlying the
duality product between $H^{s}\left(  \partial\Omega;\mathbb{C}^{3}\right)  $
and $H^{-s}\left(  \partial\Omega;\mathbb{C}^{3}\right)  $ and that
$\left\langle \mathcal{\cdot},\cdot\right\rangle _{s,\partial\Omega}$ between
$H^{s}\left(  \partial\Omega\right)  $ and $H^{-s}\left(  \partial
\Omega\right)  $
\[
\left\langle \boldsymbol{v},\boldsymbol{\ell}\right\rangle _{s,\partial\Omega
}=\sum_{i=1}^{3}\left\langle v_{i},\ell_{i}\right\rangle _{s,\partial\Omega
},\;\boldsymbol{\ell}\in H^{-s}\left(  \partial\Omega;\mathbb{C}^{3}\right)
,\boldsymbol{v}\in H^{s}\left(  \partial\Omega;\mathbb{C}^{3}\right)  .
\]

\end{corollary}

\begin{remark}
The duality $H^{s}\left(  \partial\Omega\right)  $, $H^{-s}\left(
\partial\Omega\right)  $ is usually denoted by $\left\langle \ell
,v\right\rangle _{s,\partial\Omega}$ for $\ell\in H^{-s}\left(  \partial
\Omega\right)  $ and $v\in H^{s}\left(  \partial\Omega\right)  $. The
transposition used here is convenient for the notation of the single-layer
potential of elastic waves given below.
\end{remark}

\subsection{Explicit expression for the G\"{u}nter derivatives}

Up to now, we have defined the G\"{u}nter derivatives just in the
distributional sense: $\mathcal{M}_{ij}^{\left(  \boldsymbol{n}\right)  }u\in$
$H^{s-1}\left(  \partial\Omega\right)  $ for $u\in H^{s}\left(  \partial
\Omega\right)  $, $0\leq s\leq1$. In concrete applications, $\partial\Omega$
must be considered as the boundary of a curved polyhedron. This is the case of
course when $\partial\Omega$ presents curved faces and edges, and vertices,
but also once the geometry has been effectively approximated (cf., for
example, \cite[p. 15]{Sauter-Schwab:11}). This means that $\partial\Omega$ can
be covered by a non-overlapping decomposition $\mathcal{T}$%
\[
\partial\Omega=\cup_{\omega\in\mathcal{T}}\overline{\omega}%
\]
where $\mathcal{T}$ \ is a finite family of open domains $\omega$ of
$\partial\Omega$ such that for all $\omega,\upsilon\in\mathcal{T}$,
$\omega\cap\upsilon=\emptyset$ when $\omega\neq\upsilon$. Each $\omega$ is
assumed to be a\ \ \textquotedblleft surface polygonal
domain\textquotedblright\ in the meaning that $\overline{\omega}\subset
U_{\omega}$ ($U_{\omega}$ being an open $\mathcal{C}^{\infty}$-parametrized
surface of $\mathbb{R}^{3}$), that its boundary $\partial\omega$ is a
piecewise smooth curve, and that $\omega$ is a Liptchitz domain of $U_{\omega
}$. Lipschitz domains of smooth manifolds are defined similarly to Lipchitz
domains of $\mathbb{R}^{N}$ replacing \textquotedblleft rigid
motions\textquotedblright\ in \cite[Definition 3.28]{McLean:00} by local
$\mathcal{C}^{\infty}$-diffeomorphisms onto domains of $\mathbb{R}^{2}$.
Recall that $\Omega^{+}$ is globally a Lipschitz domain, hence preventing
$\partial\Omega$ to present cusp points. Simple and widespread examples of
such boundaries are given by triangular meshes of surfaces of $\mathbb{R}^{3}%
$. Figure \ref{capsula}\ depicts a surface triangular mesh of a $C^{1,1}%
$-domain. The geometry and the mesh have been designed using the free software
\emph{Gmsh} \cite{Gezaine-Remacle:09}. For the exact surface, $\omega$ and
$U_{\omega}$ are obtained by local coordinate systems (local charts) (see, for
example, \cite{ChoquetBruhat-deWitt:91}). For the approximate surface,
$\omega$ is a triangle of $\mathbb{R}^{3}$ and $U_{\omega}$ is the plane
supporting this triangle.%

\begin{figure}[ptb]%
\centering
\includegraphics[
height=3.0554in,
width=3.0554in
]%
{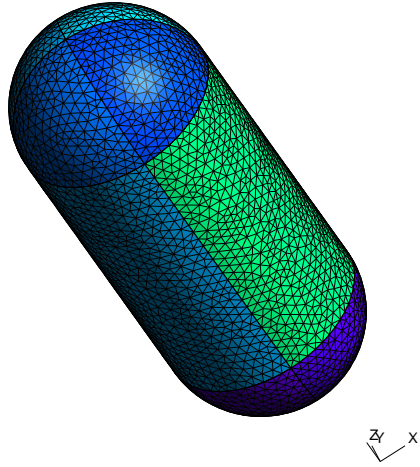}%
\caption{Polyhedral domain obtained from the surface mesh of $C^{1,1}%
$-domain.}%
\label{capsula}%
\end{figure}

Boundary element spaces are generally subspaces of the following one%
\[
\mathcal{P}_{m,\mathcal{T}}\left(  \partial\Omega\right)  =\left\{  u\in
L^{\infty}\left(  \partial\Omega\right)  ;\;u|_{\omega}\circ\Phi_{\omega}%
\in\mathbb{P}_{m},\;\forall\omega\in\mathcal{T}\right\}
\]
where $\Phi_{\omega}:D_{\omega}\subset\mathbb{R}^{2}\rightarrow U_{\omega}$ is
a local coordinate system on $U_{\omega}$ (cf., for example, \cite[p.
111]{ChoquetBruhat-deWitt:91}). Such kinds of spaces are contained in
$H^{s}\left(  \partial\Omega\right)  $ for $1/2\leq s\leq1$ if and only if
they are contained in
\[
\mathcal{C}_{\mathcal{T}}\left(  \partial\Omega\right)  =\left\{
u\in\mathcal{C}^{0}\left(  \mathcal{\partial}\Omega\right)  ;\;u|_{\omega}%
\in\mathcal{C}^{\infty}\left(  \overline{\omega}\right)  ,\;\forall\omega
\in\mathcal{T}\right\}
\]
(see, for example, \cite{Grisvard:82} when $\partial\Omega=\mathbb{R}^{2}$).

For $u\in\mathcal{C}_{\mathcal{T}}\left(  \partial\Omega\right)  $, we can
define $\mathcal{M}_{ij,\mathcal{T}}^{\left(  \boldsymbol{n}\right)  }u$
almost everywhere on $\partial\Omega$ by%
\[
\left(  \mathcal{M}_{ij,\mathcal{T}}^{\left(  \boldsymbol{n}\right)
}u\right)  |_{\omega}=\boldsymbol{\nabla}_{\omega}u|_{\omega}\times
\boldsymbol{n}\cdot\boldsymbol{e}_{i}\times\boldsymbol{e}_{j},\;\forall
\omega\in\mathcal{T}%
\]
where $\boldsymbol{\nabla}_{\omega}$ is the tangential gradient on $\omega$
and $\boldsymbol{n}$ is the unit normal on $\omega$ pointing outward from
$\Omega^{+}$. Our objective is to show that
\begin{equation}
\mathcal{M}_{ij}^{\left(  \boldsymbol{n}\right)  }u=\mathcal{M}%
_{ij,\mathcal{T}}^{\left(  \boldsymbol{n}\right)  }u\text{ for all }%
u\in\mathcal{C}_{\mathcal{T}}\left(  \partial\Omega\right)  . \label{MijT}%
\end{equation}
This identification requires some preliminaries to be established.

First, we can assume that $u|_{\omega}$ is the trace of a function $u_{\omega
}$ which is $\mathcal{C}^{\infty}$ in a neighborhood in $\mathbb{R}^{3}$ of
$U_{\omega}$. We can hence write%
\[
\boldsymbol{\nabla}_{\omega}u|_{\omega}\times\boldsymbol{n}\cdot\left(
\boldsymbol{e}_{i}\times\boldsymbol{e}_{j}\right)  =\left(  \boldsymbol{\nabla
}u_{\omega}\right)  |_{\omega}\times\boldsymbol{n}\cdot\boldsymbol{e}%
_{i}\times\boldsymbol{e}_{j}=\left(  \boldsymbol{e}_{i}\times\boldsymbol{e}%
_{j}\right)  \times\left(  \boldsymbol{\nabla}u_{\omega}\right)  |_{\omega
}\cdot\boldsymbol{n}%
\]
Since $\boldsymbol{e}_{i}\times\boldsymbol{e}_{j}\cdot\boldsymbol{e}%
_{k}=\varepsilon_{ijk}$ where $\varepsilon_{ijk}$ is the Levi-Civita symbol
($\varepsilon_{ijk}=\pm1$ if $\left\{  i,j,k\right\}  $ is an even or odd
permutation of $\left\{  1,2,3\right\}  $ respectively, and 0 otherwise),
$\boldsymbol{e}_{i}\times\boldsymbol{e}_{j}$ can be expressed in terms of its
components in the canonical basis of $\mathbb{R}^{3}$%
\[
\boldsymbol{e}_{i}\times\boldsymbol{e}_{j}=\sum_{k=1}^{3}\varepsilon
_{ijk}\boldsymbol{e}_{k}.
\]
Using the canonical identification of vector fields to 1-forms on
$\mathbb{R}^{3}$ and the Hodge star operator on $\mathbb{R}^{3}$, $\left(
\boldsymbol{e}_{i}\times\boldsymbol{e}_{j}\right)  \times\left(
\boldsymbol{\nabla}u_{\omega}\right)  |_{\omega}$ can be written as follows
\[
\left(  \boldsymbol{e}_{i}\times\boldsymbol{e}_{j}\right)  \times\left(
\boldsymbol{\nabla}u\right)  |_{\omega}=\sum_{k=1}^{3}\varepsilon_{ijk}%
\ast\left(  dx_{k}\wedge du_{\omega}\right)  |_{\omega}=\sum_{k=1}%
^{3}\varepsilon_{ijk}\ast\left(  d\left(  -u_{\omega}dx_{k}\right)  \right)
|_{\omega}%
\]

We thus retrieve the following result established component by component in
\cite{Gunter:53} without the formalism of differential forms.

\begin{lemma}
For $u\in\mathcal{C}^{\infty}\left(  \overline{\omega}\right)  $ and
$v\in\mathcal{C}_{\text{comp}}^{1}\left(  \mathbb{R}^{3}\right)  $, the
following integration by parts formula
\[
\int_{\omega}v\boldsymbol{\nabla}_{\omega}u\times\boldsymbol{n}\cdot
\boldsymbol{e}_{i}\times\boldsymbol{e}_{j}ds=-\sum_{k=1}^{3}\varepsilon
_{ijk}\int_{\partial\omega_{_{\circlearrowleft}}}uv\ dx_{k}-\int_{\omega
}u\boldsymbol{\nabla}_{\omega}v\times\boldsymbol{n}\cdot\boldsymbol{e}%
_{i}\times\boldsymbol{e}_{j}ds
\]
holds true. The orientation $\partial\omega_{_{\circlearrowleft}}$ is that
induced by $\boldsymbol{n}$.
\end{lemma}

\begin{proof}
The lemma results from the following observations%
\[
\boldsymbol{\nabla}_{\omega}uv=v\boldsymbol{\nabla}_{\omega}%
u+u\boldsymbol{\nabla}_{\omega}v
\]%
\[
\int_{\partial\omega}\boldsymbol{\nabla}_{\omega}uv\times\boldsymbol{n}%
\cdot\left(  \boldsymbol{e}_{i}\times\boldsymbol{e}_{j}\right)  ds=\sum
_{k=1}^{3}\varepsilon_{ijk}\int_{\omega}\ast\left(  d\left(  -uvdx_{k}\right)
\right)  \cdot\boldsymbol{n}ds=\sum_{k=1}^{3}\varepsilon_{ijk}\int_{\omega
}d\left(  -uvdx_{k}\right)
\]
and Stokes' formula.
\end{proof}

The following theorem gives a simple way to calculate the G\"{u}nter
derivatives when dealing with a boundary element method.

\begin{theorem}
Formula (\ref{MijT}) holds true for any $u\in\mathcal{C}_{\mathcal{T}}\left(
\partial\Omega\right)  $.
\end{theorem}

\begin{proof}
Clearly, $\mathcal{C}_{\mathcal{T}}\left(  \partial\Omega\right)  \subset
H^{1}\left(  \partial\Omega\right)  $. Hence, for $u\in\mathcal{C}%
_{\mathcal{T}}\left(  \partial\Omega\right)  $ and $v\in\mathcal{C}%
_{\text{comp}}^{\infty}\left(  \mathbb{R}^{3}\right)  $, symmetry property
(\ref{Sym_M}) yields%
\[
\left\langle v,\mathcal{M}_{ij}^{\left(  \boldsymbol{n}\right)  }%
u\right\rangle _{1-s,\partial\Omega}=\left\langle u,\mathcal{M}_{ji}^{\left(
\boldsymbol{n}\right)  }v\right\rangle _{s,\partial\Omega}=\int_{\partial
\Omega}u\mathcal{M}_{ji,\mathcal{T}}^{\left(  \boldsymbol{n}\right)  }vds.
\]
Integrating by parts, we can write
\[
\int_{\partial\Omega}u\mathcal{M}_{ji,\mathcal{T}}^{\left(  \boldsymbol{n}%
\right)  }vds=-\sum_{\omega\in\mathcal{T}}\sum_{k=1}^{3}\varepsilon_{ijk}%
\int_{\partial\omega_{\circlearrowleft}}uvdx_{k}+\int_{\partial\Omega
}v\mathcal{M}_{ij,\mathcal{T}}^{\left(  \boldsymbol{n}\right)  }uds.
\]
Since
\[
\sum_{\omega\in\mathcal{T}}\int_{\partial\omega_{\circlearrowleft}}%
uvdx_{k}=0,
\]
due to the opposite orientation on each curved edge of the non-overlapping
decomposition $\mathcal{T}$ of $\partial\Omega$, we get%
\[
\int_{\partial\Omega}v\mathcal{M}_{ij_{,}\mathcal{T}}^{\left(  \boldsymbol{n}%
\right)  }uds=\left\langle v,\mathcal{M}_{ij}^{\left(  \boldsymbol{n}\right)
}u\right\rangle _{1-s,\partial\Omega}.
\]
Formula (\ref{MijT}) then results from the density of $\mathcal{C}%
_{\text{comp}}^{\infty}\left(  \mathbb{R}^{3}\right)  $ in $H^{1-s}\left(
\partial\Omega\right)  $.
\end{proof}

\begin{remark}
The density of $\mathcal{C}_{\text{comp}}^{\infty}\left(  \mathbb{R}%
^{3}\right)  $ in $H^{s}\left(  \partial\Omega\right)  $ $(0\leq s\leq1)$, for
a Lipschitz domain $\Omega^{+}$, can be established along the same lines than
that of $\mathcal{C}_{\text{comp}}^{\infty}\left(  \mathbb{R}^{3}\right)  $ in
$L^{2}\left(  \partial\Omega\right)  $, which is proved in \cite[Th.
4.9]{Necas:12}.
\end{remark}

\section{Other expressions of the G\"{u}nter derivative matrix\label{section2}%
}

We first examine previous ways to write the G\"{u}nter derivative matrix when
$\Omega^{+}$ is of class $\mathcal{C}^{1,1}$. We then show whether or not
these expressions can be extended to a Lipschitz domain. In particular, we
recall a way to write $\mathcal{M}^{\left(  \boldsymbol{n}\right)  }$
variationally by means of a volume integral, already considered elsewhere but
not in the present context.

\subsection{Previous equivalent expressions for the G\"{u}nter derivative
matrix}

We begin with the following compact expression for the G\"{u}nter derivative
matrix given in \cite{LeLouer:14}%
\begin{equation}
\mathcal{M}^{\left(  \boldsymbol{n}\right)  }\boldsymbol{u}=\boldsymbol{\nabla
un}-\boldsymbol{n\nabla}\cdot\boldsymbol{u},\;\boldsymbol{u}\in H_{\text{loc}%
}^{2}\left(  \mathbb{R}^{3};\mathbb{C}^{3}\right)  \label{M_Comp}%
\end{equation}
which can be obtained by observing that
\[
\sum_{j=1}^{3}\mathcal{M}_{ij}^{\left(  \boldsymbol{n}\right)  }u_{j}%
=\sum_{j=1}^{3}\partial_{x_{i}}u_{j}n_{j}-n_{i}\sum_{j=1}^{3}\partial_{x_{j}%
}u_{j}.
\]
Recall that gradient $\boldsymbol{\nabla u}$ of vector $\boldsymbol{u}$ is the
matrix whose column $j$ is $\boldsymbol{\nabla}u_{j}$ $\left(  j=1,2,3\right)
.$

Probably to more clearly bring out that expression (\ref{M_Comp}) depends on
$\boldsymbol{u|}_{\partial\Omega}$ only, Le\ Lou\"{e}r \cite{LeLouer:14} used
the following way to write the gradient and the divergence on $\partial\Omega$%
\[
\boldsymbol{\nabla}u_{j}=\boldsymbol{\nabla}_{\partial\Omega}u_{j}%
+\boldsymbol{n}\partial_{\boldsymbol{n}}u_{j}\;\left(  j=1,2,3\right)
\]%
\begin{equation}
\boldsymbol{\nabla}\cdot\boldsymbol{u}=\boldsymbol{\nabla}_{\partial\Omega
}\cdot\boldsymbol{n}\times\left(  \boldsymbol{u}\times\boldsymbol{n}\right)
+2\mathcal{H}\boldsymbol{u}\cdot\boldsymbol{n}+\boldsymbol{n}\cdot
\partial_{\boldsymbol{n}}\boldsymbol{u}\text{ on }\partial\Omega,
\label{divergence}%
\end{equation}
where $\boldsymbol{\nabla}_{\partial\Omega}\cdot$ denotes the surface
divergence (see, for example, \cite[p. 72 and 75]{Nedelec:01}). We have
denoted by $2\mathcal{H}$ the mean Gaussian curvature of $\partial\Omega$,
defined as the algebraic trace $\operatorname*{tr}\mathcal{C}$ of the Gauss
curvature operator $\mathcal{C}=\boldsymbol{\nabla n}$. Formula
(\ref{divergence}) requires a domain of class $\mathcal{C}^{1,1}$ at least to
be stated. It apparently has been considered in \cite[p. 6]{LeLouer:14} for
tangential fields only (in other words, satisfying $\boldsymbol{u}%
\cdot\boldsymbol{n}=0$ on $\partial\Omega$), hence avoiding the curvature term
$2\mathcal{H}$. Defining then $\boldsymbol{\nabla}_{\partial\Omega
}\boldsymbol{u}$ as the matrix whose $j$-th column is $\boldsymbol{\nabla
}_{\partial\Omega}u_{j}$, and noting that $\boldsymbol{\nabla u}%
=\boldsymbol{\nabla}_{\partial\Omega}\boldsymbol{u}+\boldsymbol{n}\left(
\partial_{\boldsymbol{n}}\boldsymbol{u}\right)  ^{\top}$, one gets
\begin{equation}
\mathcal{M}^{\left(  \boldsymbol{n}\right)  }\boldsymbol{u}=\boldsymbol{\nabla
}_{\partial\Omega}\boldsymbol{u}-\boldsymbol{n}\left(  \boldsymbol{\nabla
}_{\partial\Omega}\cdot\boldsymbol{n}\times\left(  \boldsymbol{u}%
\times\boldsymbol{n}\right)  +2\mathcal{H}\boldsymbol{u}\cdot\boldsymbol{n}%
\right)  . \label{M_Lelouer}%
\end{equation}

There are two concerns with expression (\ref{M_Lelouer}):

\begin{itemize}
\item It involves the mean curvature $2\mathcal{H}$ of $\partial\Omega$
explicitly so that it becomes meaningless for a Lipschitz domain even when not
taking care of its derivation;

\item It does not clearly express that $\mathcal{M}^{\left(  \boldsymbol{n}%
\right)  }$ is \ a symmetric operator as stated in (\ref{Sym_M}).
\end{itemize}

With regard to the last point, one can first observe that%
\[
\boldsymbol{\nabla}_{\partial\Omega}\boldsymbol{v\ n}=\sum_{j=1}^{3}%
n_{j}\boldsymbol{\nabla}_{\partial\Omega}v_{j}=\sum_{j=1}^{3}%
\boldsymbol{\nabla}_{\partial\Omega}\left(  n_{j}v_{j}\right)  -\sum_{j=1}%
^{3}v_{j}\boldsymbol{\nabla}_{\partial\Omega}n_{j}%
\]
Since $\boldsymbol{\nabla}_{\partial\Omega}n_{j}=\boldsymbol{\nabla}%
n_{j}=\mathcal{C}_{\ast j}$, the $j$-th column of $\mathcal{C}$, and
$\mathcal{C}\boldsymbol{n}=0$, we can write
\[
\boldsymbol{\nabla}_{\partial\Omega}\boldsymbol{v\ n=\nabla}_{\partial\Omega
}\boldsymbol{v}\cdot\boldsymbol{n}-\mathcal{C}\boldsymbol{n}\times\left(
\boldsymbol{v}\times\boldsymbol{n}\right)
\]
coming, at least when $\Omega^{+}$ is $\mathcal{C}^{1,1}$-domain and
$\boldsymbol{u}\in H^{2}\left(  \mathbb{R}^{3};\mathbb{C}^{3}\right)  $, to
the following way to write the G\"{u}nter derivative matrix
\begin{equation}
\mathcal{M}^{\left(  \boldsymbol{n}\right)  }\boldsymbol{u}=\boldsymbol{\nabla
}_{\partial\Omega}\boldsymbol{u}\cdot\boldsymbol{n}-\boldsymbol{n\nabla
}_{\partial\Omega}\cdot\boldsymbol{n}\times\left(  \boldsymbol{u}%
\times\boldsymbol{n}\right)  -\mathcal{C}\left(  \boldsymbol{n}\times\left(
\boldsymbol{u}\times\boldsymbol{n}\right)  \right)  -2\mathcal{H}%
\boldsymbol{u}\cdot\boldsymbol{n\ n} \label{M_courbure}%
\end{equation}
more clearly expressing the symmetry properties stated above.

We now come to the expression of the G\"{u}nter derivative matrix most often
used to express the traction in Lam\'{e} static elasticity \cite[formula
(1.14) p. 282]{Kupradze:79}%
\begin{equation}
\mathcal{M}^{\left(  \boldsymbol{n}\right)  }\boldsymbol{u}=\partial
_{\boldsymbol{n}}\boldsymbol{u}+\boldsymbol{n}\times\boldsymbol{\nabla}%
\times\boldsymbol{u}-\boldsymbol{n\nabla}\cdot\boldsymbol{u}. \label{Kupradze}%
\end{equation}
Since the derivation of this formula does not seem to have been explicitly
carried out before, for the convenience of the reader, we show how it can be
established from the above compact expression of $\mathcal{M}^{\left(
\boldsymbol{n}\right)  }$. Writing%
\[
\mathcal{M}^{\left(  \boldsymbol{n}\right)  }\boldsymbol{u}=\left(
\boldsymbol{\nabla u}-\boldsymbol{\nabla u}^{\top}\right)  \boldsymbol{n}%
+\boldsymbol{\nabla u}^{\top}\boldsymbol{n}-\boldsymbol{n\nabla}%
\cdot\boldsymbol{u},
\]
we get%
\[
\mathcal{M}^{\left(  \boldsymbol{n}\right)  }\boldsymbol{u}=\partial
_{\boldsymbol{n}}\boldsymbol{u}+\left(  \boldsymbol{\nabla u}%
-\boldsymbol{\nabla u}^{\top}\right)  \boldsymbol{n}-\boldsymbol{n\nabla}%
\cdot\boldsymbol{u}.
\]
Now
\[
\left(  \boldsymbol{\nabla u}-\boldsymbol{\nabla u}^{\top}\right)
_{ij}=\partial_{x_{i}}u_{j}-\partial_{x_{j}}u_{i}=\sum_{l,m=1}^{3}\left(
\delta_{il}\delta_{jm}-\delta_{im}\delta_{jl}\right)  \partial_{x_{l}}u_{m}.
\]
Using the elementary writing of $\delta_{il}\delta_{jm}-\delta_{im}\delta
_{jl}$ in terms of the Levi-Civita symbol%
\[
\delta_{il}\delta_{jm}-\delta_{im}\delta_{jl}=\sum_{k=1}^{3}\varepsilon
_{ijk}\varepsilon_{lmk}%
\]
we come to%
\begin{align*}
\left(  \boldsymbol{\nabla u}-\boldsymbol{\nabla u}^{\top}\right)  _{ij}  &
=\sum_{l,m,k=1}^{3}\varepsilon_{ijk}\varepsilon_{lmk}\partial_{x_{l}}u_{m}\\
&  =\sum_{k=1}^{3}\varepsilon_{ijk}\left(  \boldsymbol{\nabla}\times
\boldsymbol{u}\right)  _{k},
\end{align*}
and thus to
\[
\left(  \left(  \boldsymbol{\nabla u}-\boldsymbol{\nabla u}^{\top}\right)
\boldsymbol{n}\right)  _{i}=\sum_{j,k=1}^{3}\varepsilon_{ijk}n_{j}\left(
\boldsymbol{\nabla}\times\boldsymbol{u}\right)  _{k}=\left(  \boldsymbol{n}%
\times\boldsymbol{\nabla}\times\boldsymbol{u}\right)  _{i}\;\left(
i=1,2,3\right)  .
\]

Form (\ref{Kupradze}) of $\mathcal{M}^{\left(  \boldsymbol{n}\right)
}\boldsymbol{u}$ gives rise to two concerns also:

\begin{itemize}
\item At least, in a direct way, it can not be evaluated from $\boldsymbol{u|}%
_{\partial\Omega}$ only;

\item Contrary to (\ref{M_courbure}), it keeps a meaning when $\Omega^{+}$ is
only a Lipschitz domain but requires that $\boldsymbol{u}\in H^{2}\left(
\mathbb{R}^{3};\mathbb{C}^{3}\right)  $ to be defined.
\end{itemize}

With regard to the first of the above two points, Darbas and Le\ Lou\"{e}r
\cite{Darbas-Lelouer:15}\ used expression (\ref{divergence}) \ for the
divergence \cite[Formula (2.5.215)]{Nedelec:01} together with the following
one for the curl
\[
\boldsymbol{n}\times\boldsymbol{\nabla}\times\boldsymbol{u}=\boldsymbol{\nabla
}_{\partial\Omega}\boldsymbol{u}\cdot\boldsymbol{n}-\mathcal{C}\boldsymbol{n}%
\times\left(  \boldsymbol{u}\times\boldsymbol{n}\right)  -\boldsymbol{n}%
\times\left(  \partial_{\boldsymbol{n}}\boldsymbol{u}\times\boldsymbol{n}%
\right)
\]
\cite[Formula (2.5.225)]{Nedelec:01} to get formula (\ref{M_courbure}) from
formula (\ref{Kupradze}).

\subsection{Expression of the G\"{u}nter derivative matrix by a volume
integral}

The trace $\partial_{\boldsymbol{n}}\boldsymbol{u}+\boldsymbol{n}%
\times\boldsymbol{\nabla}\times\boldsymbol{u}-\boldsymbol{n\nabla}%
\cdot\boldsymbol{u}$ has been considered in \cite[Proof of Lemma 2.1 p.
248]{Costabel-Dauge:99} without any reference to the G\"{u}nter derivatives.
More particularly, collecting some formulae in this paper, we readily come to
the following Green formula%
\begin{equation}%
{\displaystyle\int_{\Omega^{\pm}}}
\boldsymbol{\nabla u}\cdot\boldsymbol{\nabla v}-\boldsymbol{\nabla}%
\times\boldsymbol{u}\cdot\boldsymbol{\nabla}\times\boldsymbol{v}%
-\boldsymbol{\nabla}\cdot\boldsymbol{u\ \nabla}\cdot\boldsymbol{v\ }dx=\pm%
{\displaystyle\int_{\partial\Omega}}
\left(  \partial_{\boldsymbol{n}}\boldsymbol{u}+\boldsymbol{n}\times
\boldsymbol{\nabla}\times\boldsymbol{u}-\boldsymbol{n\nabla}\cdot
\boldsymbol{u}\right)  \cdot\boldsymbol{v}\ ds \label{M_Green}%
\end{equation}
for $\boldsymbol{u}$ and $\boldsymbol{v}$ in $H^{2}\left(  \mathbb{R}%
^{3};\mathbb{C}^{3}\right)  $ where the bilinear form underlying the scalar
product of two $3\times3$ matrices is defined by
\[
\boldsymbol{\nabla u}\cdot\boldsymbol{\nabla v}=\sum_{j=1}^{3}%
\boldsymbol{\nabla}u_{j}\cdot\boldsymbol{\nabla}v_{j}=\sum_{i,j=1}^{3}%
\partial_{x_{i}}u_{j}\partial_{x_{i}}v_{j}.
\]
It is assumed there that $\Omega^{+}$ is a curved polyhedron but the
derivation remains valid when $\Omega^{+}$ is a Lipschitz domain and for
$\boldsymbol{v}$ in $H^{1}\left(  \Omega^{\pm};\mathbb{C}^{3}\right)  $. In
the same way, the above Green formula is still holding true for
$\boldsymbol{u}\in H_{\text{loc}}^{2}\left(  \mathbb{R}^{3};\mathbb{C}%
^{3}\right)  $ and $\boldsymbol{v}\in H_{\text{comp}}^{1}\left(
\mathbb{R}^{3};\mathbb{C}^{3}\right)  $ or $\boldsymbol{u}\in H_{\text{comp}%
}^{2}\left(  \mathbb{R}^{3};\mathbb{C}^{3}\right)  $ and $\boldsymbol{v}\in
H_{\text{loc}}^{1}\left(  \mathbb{R}^{3};\mathbb{C}^{3}\right)  $. Actually,
formula (\ref{M_Green}) can also be directly deduced from an older Green
formula considered in \cite[p. 220]{Knauff-Kress:79}%
\begin{equation}%
{\displaystyle\int_{\Omega^{\pm}}}
\boldsymbol{\Delta u}\cdot\boldsymbol{v}+\boldsymbol{\nabla}\times
\boldsymbol{u}\cdot\boldsymbol{\nabla}\times\boldsymbol{v}+\boldsymbol{\nabla
}\cdot\boldsymbol{u\ \nabla}\cdot\boldsymbol{v\ }dx=\pm%
{\displaystyle\int_{\partial\Omega}}
\left(  \boldsymbol{\nabla}\times\boldsymbol{u}\times\boldsymbol{n}%
+\boldsymbol{n\nabla}\cdot\boldsymbol{u}\right)  \cdot\boldsymbol{v\ }ds.
\label{Knauff-Kress}%
\end{equation}

We then directly come to the following theorem giving the expression of the
G\"{u}nter derivative matrix in terms of a volume integral.

\begin{theorem}
Let $\Omega^{+}$ be a bounded Lipschitz domain of $\mathbb{R}^{3}$. Using the
general notation introduced above, we have
\begin{equation}
\left\langle \boldsymbol{v},\mathcal{M}^{\left(  \boldsymbol{n}\right)
}\boldsymbol{u}\right\rangle _{1/2,\partial\Omega}=\pm\int_{\Omega^{\pm}%
}\boldsymbol{\nabla u}\cdot\boldsymbol{\nabla v}-\boldsymbol{\nabla}%
\times\boldsymbol{u}\cdot\boldsymbol{\nabla}\times\boldsymbol{v}%
-\boldsymbol{\nabla}\cdot\boldsymbol{u\ \nabla}\cdot\boldsymbol{v\ }dx
\label{M_vol}%
\end{equation}
for $\boldsymbol{u}\in H_{\text{loc}}^{1}\left(  \mathbb{R}^{3};\mathbb{C}%
^{3}\right)  $ and $\boldsymbol{v}\in H_{\text{comp}}^{1}\left(
\mathbb{R}^{3};\mathbb{C}^{3}\right)  $.
\end{theorem}

\begin{proof}
In view of (\ref{M_Green}), assuming that $\boldsymbol{v}\in H_{\text{comp}%
}^{2}\left(  \mathbb{R}^{3};\mathbb{C}^{3}\right)  $, we can write%
\[
\int_{\partial\Omega}\boldsymbol{u}\cdot\mathcal{M}^{\left(  \boldsymbol{n}%
\right)  }\boldsymbol{v}\ ds=\pm\int_{\Omega^{\pm}}\boldsymbol{\nabla u}%
\cdot\boldsymbol{\nabla v}-\boldsymbol{\nabla}\times\boldsymbol{u}%
\cdot\boldsymbol{\nabla}\times\boldsymbol{v}-\boldsymbol{\nabla}%
\cdot\boldsymbol{u\ \nabla}\cdot\boldsymbol{v\ }dx.
\]
Noting then that
\[
\int_{\partial\Omega}\boldsymbol{u}\cdot\mathcal{M}^{\left(  \boldsymbol{n}%
\right)  }\boldsymbol{v}\ ds=\left\langle \boldsymbol{u},\mathcal{M}^{\left(
\boldsymbol{n}\right)  }\boldsymbol{v}\right\rangle _{1/2,\partial\Omega
}=\left\langle \boldsymbol{v},\mathcal{M}^{\left(  \boldsymbol{n}\right)
}\boldsymbol{u}\right\rangle _{1/2,\partial\Omega}%
\]
we get (\ref{M_vol}) for $\boldsymbol{v}\in H_{\text{comp}}^{2}\left(
\mathbb{R}^{3};\mathbb{C}^{3}\right)  $. The proof can then be readily
completed from the density of $H_{\text{comp}}^{2}\left(  \mathbb{R}%
^{3};\mathbb{C}^{3}\right)  $ in $H_{\text{comp}}^{1}\left(  \mathbb{R}%
^{3};\mathbb{C}^{3}\right)  $.
\end{proof}

\section{Application to the elastic wave boundary-layer
potentials\label{section3}}

In this section, we extend the regularization of elastic wave boundary-layer
potentials devised by Le~Lou\"{e}r \cite{LeLouer:14,Lelouerv1:12} for a
geometry of class $\mathcal{C}^{2}$ to the case of a Lipschitz domain. This
extension is straightforward for the traces of the single- and the
double-layer potentials. We just more explicitly bring out an intermediary
expression for the double-layer potential and an identity linking the elastic
wave boundary-layer potentials to those related to the Helmholtz equation. We
focus on the traction of the double-layer potential which requires a different
technique of proof. Meanwhile, as an application of these regularization
techniques, we show how the mapping properties of the elastic waves potentials
easily reduce to those related to the Helmholtz equation without resorting to
the general theory of boundary layer potentials for elliptic systems.

\subsection{Layer potentials of elastic waves}

For $\boldsymbol{p}\in H^{-1/2}\left(  \partial\Omega;\mathbb{C}^{3}\right)
$, the elastic wave single-layer potential can be expressed as follows
\[
S\boldsymbol{p}\left(  x\right)  =\left\langle \Gamma\left(  x,y\right)
,\boldsymbol{p}_{y}\right\rangle _{1/2,\partial\Omega}\;\left(  x\in\Omega
^{+}\cup\Omega^{-}\right)  ,
\]
in terms of the Kupradze matrix $\Gamma\left(  x,y\right)  $ whose entries are
given by \cite[p. 85]{Kupradze:79}%
\[
\Gamma_{kl}\left(  x,y\right)  =\frac{1}{\omega^{2}\varrho}\left(  \kappa
_{s}^{2}G_{\kappa_{s}}\left(  x,y\right)  \delta_{kl}+\partial_{x_{k}}%
\partial_{x_{l}}\left(  G_{\kappa_{p}}-G_{\kappa_{s}}\right)  \left(
x,y\right)  \right)  \;\left(  k,l=1,2,3\right)  .
\]
Dummy variable $y$ is used to indicate that the duality brackets link
$\boldsymbol{p}$ to the function $y\rightarrow\Gamma\left(  x,y\right)  $
indexed by parameter $x$. The notation $\left\langle \Gamma\left(  x,y\right)
,\boldsymbol{p}_{y}\right\rangle _{1/2,\partial\Omega}$ refers to the vector
whose component $k$ is given by
\[
\sum_{l=1}^{3}\left\langle \Gamma_{kl}\left(  x,y\right)  ,\left(
p_{l}\right)  _{y}\right\rangle _{1/2,\partial\Omega}%
\]
where $\left(  p_{l}\right)  _{y}$ is component $l$ of $\boldsymbol{p}_{y}$.
It is this formula that motivates the transposition in the duality brackets
$H^{1/2}\left(  \partial\Omega\right)  $, $H^{-1/2}\left(  \partial
\Omega\right)  $ adopted above. As usual
\[
\kappa_{p}=\omega\sqrt{\varrho/\left(  2\mu+\lambda\right)  }\text{ and
}\kappa_{s}=\omega\sqrt{\varrho/\mu}%
\]
are the wavenumbers corresponding to compression or P-waves and shear or
S-waves respectively. The constants $\omega$, $\varrho$, $\mu>0$ and
$\lambda\geq0$ characterize the angular frequency of the wave, the density and
the Lam\'{e} coefficients of the elastic medium respectively. Finally,
$G_{\kappa}\left(  x,y\right)  =\exp\left(  i\kappa\left\vert x-y\right\vert
\right)  /4\pi\left\vert x-y\right\vert $ is the Green kernel characterizing
the solutions of the Helmholtz equation
\[
\Delta_{y}G_{\kappa}\left(  x,y\right)  +\kappa^{2}G_{\kappa}\left(
x,y\right)  =-\delta_{x}\text{ in }\mathcal{D}^{\prime}\left(  \mathbb{R}%
^{3}\right)  ,
\]
satisfying the Sommerfeld radiation condition
\[
\lim_{\left\vert y\right\vert \rightarrow\infty}\left\vert y\right\vert
\left(  \partial_{\left\vert y\right\vert }G_{\kappa}\left(  x,y\right)
-i\kappa G_{\kappa}\left(  x,y\right)  \right)  =0,
\]
$\delta_{x}$ being the Dirac mass at $x$.

Actually, we think that it is more convenient to express $S\boldsymbol{p}$ in
terms of the Helmholtz equation single-layer potentials $V_{\kappa_{p}%
}\boldsymbol{p}$ and $V_{\kappa_{s}}\boldsymbol{p}$ characterizing the P- and
the S-waves respectively
\begin{equation}
S\boldsymbol{p}=\frac{1}{\omega^{2}\varrho}\left(  \kappa_{s}^{2}V_{\kappa
_{s}}\boldsymbol{p}+\boldsymbol{\nabla\nabla}\cdot\left(  V_{\kappa_{s}%
}-V_{\kappa_{p}}\right)  \boldsymbol{p}\right)  , \label{EW-S}%
\end{equation}
where generically the single-layer potential related to the Helmholtz equation
corresponding to the wave number $\kappa>0$ is defined by
\[
\left(  V_{\kappa}\boldsymbol{p}\left(  x\right)  \right)  _{\ell
}=\left\langle G_{\kappa}\left(  x,y\right)  ,\left(  p_{\ell}\right)
_{y}\right\rangle _{1/2,\partial\Omega},\quad x\in\Omega^{+}\cup\Omega^{-},
\]
$\left(  V_{\kappa}\boldsymbol{p}\left(  x\right)  \right)  _{\ell}$ $\left(
\ell=1,2,3\right)  $ being the $\ell$-th component of $V_{\kappa
}\boldsymbol{p}\left(  x\right)  $.

The following proposition recalls some important properties of these potentials.

\begin{proposition}
For $p\in H^{-1/2+s}\left(  \partial\Omega\right)  $, $V_{\kappa}p\in
H_{\text{loc}}^{1+s}\left(  \mathbb{R}^{3}\right)  $, $-1/2\leq s\leq1/2$. It
satisfies the Helmholtz equation $\Delta V_{\kappa}p+\kappa^{2}V_{\kappa}p=0$
in $\Omega^{+}\cup\Omega^{-}$ and the Sommerfeld radiation condition. Moreover%
\begin{equation}
\left(  V_{\kappa_{s}}-V_{\kappa_{p}}\right)  p\in H_{\text{loc}}^{3+s}\left(
\mathbb{R}^{3}\right)  . \label{reg_H3}%
\end{equation}

\end{proposition}

\begin{proof}
The mapping property of $V_{\kappa}$ is a particular case of that of
single-layer potentials of more general elliptic equations (cf., for example,
\cite[Th.~1]{Costabel:88-2} or \cite[Th.~6.11]{McLean:00}). The fact that it
satisfies the Helmholtz equation and the Sommerfeld radiation condition is
stated for example in \cite[p.~117]{Nedelec:01}. The final property is
well-known. For the convenience of the reader, we prove it below. From the
definition of $V_{\kappa}$ (cf., for example, \cite[p.~201]{McLean:00}), we
can write
\[
\Delta V_{\kappa}p+\kappa^{2}V_{\kappa}p=-p\delta_{\partial\Omega}%
\]
where $p\delta_{\partial\Omega}$ is the single-layer distribution defined by%
\[
\left\langle \varphi,p\delta_{\partial\Omega}\right\rangle _{\mathcal{D}%
,\mathcal{D}^{\prime}}=\left\langle \varphi|_{\partial\Omega},p\right\rangle
_{1/2,\partial\Omega},\;\varphi\in\mathcal{D}\left(  \mathbb{R}^{3}\right)
\]
where $\left\langle \cdot,\cdot\right\rangle _{\mathcal{D},\mathcal{D}%
^{\prime}}$ is the bilinear form underlying the duality brackets
$\mathcal{D}\left(  \mathbb{R}^{3}\right)  $, $\mathcal{D}^{\prime}\left(
\mathbb{R}^{3}\right)  $. Thus%
\[
\Delta\left(  V_{\kappa_{s}}-V_{\kappa_{p}}\right)  p=\kappa_{p}^{2}%
V_{\kappa_{p}}p-\kappa_{s}^{2}V_{\kappa_{s}}p.
\]
Property (\ref{reg_H3}) is then a direct consequence of the interior
regularity for the solutions of the elliptic equations (see, for example,
\cite[Th.~4.16]{McLean:00}).
\end{proof}

The double-layer potential for elastic waves is defined for $\boldsymbol{\psi
}\in H^{1/2}\left(  \partial\Omega;\mathbb{C}^{3}\right)  $ by \cite[p.~301]%
{Kupradze:79}%
\[
K\boldsymbol{\psi}\left(  x\right)  =-\int_{\partial\Omega}\left(
T_{y}^{\left(  \boldsymbol{n}\right)  }\Gamma\left(  x,y\right)  \right)
^{\top}\boldsymbol{\psi}\left(  y\right)  \ ds_{y},\quad x\in\Omega^{+}%
\cup\Omega^{-},
\]
where $T_{y}^{\left(  \boldsymbol{n}\right)  }$ denotes the traction operator
defined for $\boldsymbol{u}\in H_{\text{loc}}^{2}\left(  \mathbb{R}%
^{3};\mathbb{C}^{3}\right)  $ by%
\[
T^{\left(  \boldsymbol{n}\right)  }\boldsymbol{u}=2\mu\partial_{\boldsymbol{n}%
}\boldsymbol{u}+\lambda\boldsymbol{n\nabla}\cdot\boldsymbol{u}+\mu
\boldsymbol{n}\times\boldsymbol{\nabla}\times\boldsymbol{u},
\]
$T_{y}^{\left(  \boldsymbol{n}\right)  }\Gamma\left(  x,y\right)  $ being the
matrix whose column $j$ is obtained by applying $T_{y}^{\left(  \boldsymbol{n}%
\right)  }$ to column $j$ of $\Gamma\left(  x,y\right)  $. The reader must
take care of the fact that the above double-layer as well as the one
associated with the Helmholtz equation
\[
N_{\kappa}\lambda\left(  x\right)  =-\int_{\partial\Omega}\partial
_{\boldsymbol{n}_{y}}G_{\kappa}\left(  x,y\right)  \lambda\left(  y\right)
ds_{y}\quad x\in\Omega^{+}\cup\Omega^{-},
\]
are of the opposite sign of those considered in the literature (cf.
\cite[p.~301]{Kupradze:79} and \cite[Formulae (2.2.19) and (1.2.2)]%
{Hsiao-Wendland:08}). We find this notation more compatible with the formulae
expressing the jump of the traction of the single-layer potential for elastic
waves and the normal derivative of the Helmholtz equation single-layer potential.

The above extension of $\mathcal{M}^{\left(  \boldsymbol{n}\right)  }$ to a
Lipschitz domain allows us to do the same for the expressions of the
double-layer potential devised by Le\ Lou\"{e}r \cite{LeLouer:14} for
$\mathcal{C}^{2}$-domains.

\begin{proposition}
The double-layer potential can be expressed as%
\begin{equation}
K\boldsymbol{\psi}=\boldsymbol{\nabla}V_{\kappa_{p}}\boldsymbol{n}%
\cdot\boldsymbol{\psi}-\boldsymbol{\nabla}\times V_{\kappa_{s}}\boldsymbol{n}%
\times\boldsymbol{\psi}-2\mu S\mathcal{M}^{\left(  \boldsymbol{n}\right)
}\boldsymbol{\psi}\text{ in }\Omega^{+}\cup\Omega^{-}. \label{EW-K-I}%
\end{equation}
Moreover, in view of the following identity%
\begin{equation}
\boldsymbol{\nabla}V_{\kappa_{s}}\boldsymbol{n}\cdot\boldsymbol{\psi
}-\boldsymbol{\nabla}\times V_{\kappa_{s}}\boldsymbol{n}\times\boldsymbol{\psi
}=N_{\kappa_{s}}\boldsymbol{\psi}+V_{\kappa_{s}}\mathcal{M}^{\left(
\boldsymbol{n}\right)  }\boldsymbol{\psi}\text{ in }\Omega^{+}\cup\Omega^{-}
\label{EW-K-half}%
\end{equation}
it can be put also in the following form%
\begin{equation}
K\boldsymbol{\psi}=N_{\kappa_{s}}\boldsymbol{\psi}+\left(  V_{\kappa_{s}}-2\mu
S\right)  \mathcal{M}^{\left(  \boldsymbol{n}\right)  }\boldsymbol{\psi
}+\boldsymbol{\nabla}\left(  V_{\kappa_{p}}-V_{\kappa_{s}}\right)
\boldsymbol{n}\cdot\boldsymbol{\psi}\text{ in }\Omega^{+}\cup\Omega^{-}.
\label{EW-K-II}%
\end{equation}

\end{proposition}

\begin{proof}
Both the above expressions of $K\boldsymbol{\psi}$ are straightforward
extensions of calculations carried out in \cite{LeLouer:14}. Formulae
(\ref{EW-K-I}) and (\ref{EW-K-half}) are stated here in their own right
instead of being parts of the calculations.
\end{proof}

The following theorem can then be proved in an elementary fashion from the
properties of the Helmholtz equation layer potentials.

\begin{theorem}
The elastic wave layer potentials have the following mapping properties:%
\[%
\begin{array}
[c]{l}%
S:H^{-1/2+s}\left(  \partial\Omega;\mathbb{C}^{3}\right)  \rightarrow
H_{\text{loc}}^{1+s}\left(  \mathbb{R}^{3};\mathbb{C}^{3}\right) \\
K:H^{1/2+s}\left(  \partial\Omega;\mathbb{C}^{3}\right)  \rightarrow
H_{\text{loc}}^{1+s}\left(  \overline{\Omega^{\pm}};\mathbb{C}^{3}\right)
\end{array}
\text{ for }-1/2<s<1/2;
\]
The potentials $\boldsymbol{u}=S\boldsymbol{p}$ or $\boldsymbol{u}%
=K\boldsymbol{\psi}$ satisfy%
\[
\left\{
\begin{array}
[c]{l}%
\Delta^{\ast}\boldsymbol{u}+\omega^{2}\varrho\boldsymbol{u}=0\text{ in }%
\Omega^{+}\cup\Omega^{-},\\
\boldsymbol{u}\text{ fulfils the Kupradze radiation conditions \cite[p.~124]%
{Kupradze:79} }%
\end{array}
\right.
\]
where $\Delta^{\ast}$ is the elastic laplacian given by%
\[
\Delta^{\ast}\boldsymbol{u}=\mu\Delta\boldsymbol{u}+\left(  \lambda
+\mu\right)  \boldsymbol{\nabla\nabla}\cdot\boldsymbol{u}.
\]

\end{theorem}

\begin{proof}
The first part of the proof follows from Costabel's results on mapping
properties of scalar elliptic operators \cite{Costabel:88-2}. The second one
is obtained by straightforward calculations from (\ref{EW-S}) and
(\ref{EW-K-I}).
\end{proof}

\subsection{Traces of elastic wave layer potentials}

The traces of the single- and double-layer potentials $S$ and $K$ and their
mapping properties can also be deduced from the traces of the layer potentials
of the Helmholtz equation.

\begin{theorem}
The operators defined by $\left(  S\boldsymbol{p}\right)  ^{\pm}$ for
$\boldsymbol{p}\in H^{-1/2}\left(  \partial\Omega;\mathbb{C}^{3}\right)
$\ and $\left(  K\boldsymbol{\psi}\right)  ^{\pm}$ for $\boldsymbol{\psi}\in
H^{1/2}( \partial\Omega;\allowbreak\mathbb{C}^{3}) $\ have the following
expressions
\begin{align*}
\left(  S\boldsymbol{p}\right)  ^{\pm}  &  =\frac{1}{\omega^{2}\varrho}\left(
\kappa_{s}^{2}V_{\kappa_{s}}\boldsymbol{p}+\boldsymbol{\nabla\nabla}%
\cdot\left(  V_{\kappa_{s}}-V_{\kappa_{p}}\right)  \boldsymbol{p}\right)  ,\\
\left(  K\boldsymbol{\psi}\right)  ^{\pm}  &  =\left(  N_{\kappa_{s}%
}\boldsymbol{\psi}\right)  ^{\pm}+\left(  V_{\kappa_{s}}-2\mu S\right)
\mathcal{M}^{\left(  \boldsymbol{n}\right)  }\boldsymbol{\psi}%
+\boldsymbol{\nabla}\left(  V_{\kappa_{p}}-V_{\kappa_{s}}\right)
\boldsymbol{n}\cdot\boldsymbol{\psi}.
\end{align*}
In particular, the jumps of the related potentials are given by%
\[
\left[  S\boldsymbol{p}\right]  =(S\boldsymbol{p})^{+}-(S\boldsymbol{p}%
)^{-}=0,\quad\left[  K\boldsymbol{\psi}\right]  =\boldsymbol{\psi}\text{.}%
\]
As a result, we simply refer to $\left(  S\boldsymbol{p}\right)  ^{\pm}$ by
$S\boldsymbol{p}$ below. \newline The mapping properties of these operators
are given, for $-1/2<s<1/2$, by%
\begin{align*}
S\boldsymbol{p}  &  \in H^{1/2+s}\left(  \partial\Omega;\mathbb{C}^{3}\right)
,\text{ for }\boldsymbol{p}\in H^{-1/2+s}\left(  \partial\Omega;\mathbb{C}%
^{3}\right)  ;\\
\left(  K\boldsymbol{\psi}\right)  ^{\pm}  &  \in H^{1/2+s}\left(
\partial\Omega;\mathbb{C}^{3}\right)  ,\text{ for }\boldsymbol{\psi}\in
H^{1/2+s}\left(  \partial\Omega;\mathbb{C}^{3}\right)  .
\end{align*}

\end{theorem}

\begin{proof}
The only point requiring some care concerns the term $\boldsymbol{\nabla
}\left(  V_{\kappa_{s}}-V_{\kappa_{p}}\right)  \boldsymbol{n}\cdot
\boldsymbol{\psi}$. But since $\boldsymbol{\psi}\in H^{1/2+s}\left(
\partial\Omega;\mathbb{C}^{3}\right)  $, $\boldsymbol{n}\cdot\boldsymbol{\psi
}$ is well-defined in $L^{2}\left(  \partial\Omega\right)  $ and thus belongs
to $H^{-1/2+s}\left(  \partial\Omega\right)  $ for $-1/2<s<1/2$. Regularity
property (\ref{reg_H3}) then yields
\[
\boldsymbol{\nabla}\left(  V_{\kappa_{s}}-V_{\kappa_{p}}\right)
\boldsymbol{n}\cdot\boldsymbol{\psi}\in H_{\text{loc}}^{s+2}\left(
\mathbb{R}^{3};\mathbb{C}^{3}\right)  \text{ for }-1/2<s<1/2\text{. }%
\]
This is enough to define its trace in $H^{1/2+s}\left(  \partial
\Omega;\mathbb{C}^{3}\right)  $ from Costabel's extension of the trace theorem
for Lipschitz domains (see \cite[Lemma~3.6]{Costabel:88-2} and \cite[Th.~3.38]%
{McLean:00}).
\end{proof}

\begin{remark}
The point preventing the extension of the above mapping properties to the
end-points, $s=\pm1/2$, concerns Costabel's extension of the trace theorem
from $H_{\text{loc}}^{s}\left(  \overline{\Omega^{\pm}}\right)  $ onto
$H^{s-1/2}\left(  \partial\Omega\right)  $, valid only for $1/2<s<3/2$.
\end{remark}

\subsection{Traction of the elastic waves layer potentials}

We begin with the following classical lemma which defines the traction
$T^{\left(  \boldsymbol{n}\right)  }\boldsymbol{u}$ for $\boldsymbol{u}$ in
the following space
\[
H_{\text{loc}}^{1}\left(  \boldsymbol{\Delta}^{\ast},\overline{\Omega^{\pm}%
}\right)  =\left\{  \boldsymbol{v}\in H_{\text{loc}}^{1}\left(  \overline
{\Omega^{\pm}};\mathbb{C}^{3}\right)  ;\;\boldsymbol{\Delta}^{\ast
}\boldsymbol{v}\in L_{\text{loc}}^{2}\left(  \overline{\Omega^{\pm}%
};\mathbb{C}^{3}\right)  \right\}  .
\]
Meanwhile, we adapt previous expressions of this operator, written in terms of
the G\"{u}nter derivative matrix, to the present context of a Lipschitz geometry.

\begin{lemma}
\label{Lemma3}For $\boldsymbol{u}\in H_{\text{loc}}^{1}\left(
\boldsymbol{\Delta}^{\ast},\overline{\Omega^{\pm}}\right)  $ and
$\boldsymbol{v}\in H_{\text{comp}}^{1}\left(  \mathbb{R}^{3};\mathbb{C}%
^{3}\right)  $, the following formula defines $\left(  T^{\left(
\boldsymbol{n}\right)  }\boldsymbol{u}\right)  ^{\pm}$ in $H^{-1/2}\left(
\partial\Omega;\mathbb{C}^{3}\right)  $%
\begin{equation}%
{\displaystyle\int_{\Omega^{\pm}}}
2\mu\boldsymbol{\nabla u}\cdot\boldsymbol{\nabla v}-\mu\boldsymbol{\nabla
}\times\boldsymbol{u}\cdot\boldsymbol{\nabla}\times\boldsymbol{v}%
+\lambda\boldsymbol{\nabla}\cdot\boldsymbol{u}\ \boldsymbol{\nabla}%
\cdot\boldsymbol{v}+\boldsymbol{\Delta}^{\ast}\boldsymbol{u}\cdot
\boldsymbol{v\ }dx=\left\langle \boldsymbol{v},\pm\left(  T^{\left(
\boldsymbol{n}\right)  }\boldsymbol{u}\right)  ^{\pm}\right\rangle
_{1/2,-1/2}.\label{traction-var}%
\end{equation}
Moreover, the traction $\left(  T^{\left(  \boldsymbol{n}\right)
}\boldsymbol{u}\right)  ^{\pm}$ can also be expressed in either of the two
following forms
\begin{equation}%
{\displaystyle\int_{\Omega^{\pm}}}
\mu\boldsymbol{\nabla}\times\boldsymbol{u}\cdot\boldsymbol{\nabla}%
\times\boldsymbol{v}+\left(  \lambda+2\mu\right)  \boldsymbol{\nabla}%
\cdot\boldsymbol{u}\ \boldsymbol{\nabla}\cdot\boldsymbol{v}+\boldsymbol{\Delta
}^{\ast}\boldsymbol{u}\cdot\boldsymbol{v\ }dx+\left\langle \boldsymbol{v}%
,\pm2\mu\mathcal{M}^{\left(  \boldsymbol{n}\right)  }\boldsymbol{u}^{\pm
}\right\rangle _{1/2,\partial\Omega}=\left\langle \boldsymbol{v},\pm\left(
T^{\left(  \boldsymbol{n}\right)  }\boldsymbol{u}\right)  ^{\pm}\right\rangle
_{1/2,\partial\Omega},\label{traction-M1}%
\end{equation}%
\begin{equation}%
{\displaystyle\int_{\Omega^{\pm}}}
\mu\boldsymbol{\nabla u}\cdot\boldsymbol{\nabla v}+\left(  \lambda+\mu\right)
\boldsymbol{\nabla}\cdot\boldsymbol{u}\ \boldsymbol{\nabla}\cdot
\boldsymbol{v}+\boldsymbol{\Delta}^{\ast}\boldsymbol{u}\cdot\boldsymbol{v\ }%
dx+\left\langle \boldsymbol{v},\pm\mu\mathcal{M}^{\left(  \boldsymbol{n}%
\right)  }\boldsymbol{u}^{\pm}\right\rangle _{1/2,\partial\Omega}=\left\langle
\boldsymbol{v},\pm\left(  T^{\left(  \boldsymbol{n}\right)  }\boldsymbol{u}%
\right)  ^{\pm}\right\rangle _{1/2,\partial\Omega}.\label{traction-M2}%
\end{equation}

\end{lemma}

\begin{proof}
Identity (\ref{traction-var}) is obtained from identity (\ref{Knauff-Kress})
and usual Green formula for $\boldsymbol{u}\in H_{\text{loc}}^{2}\left(
\overline{\Omega^{\pm}};\mathbb{C}^{3}\right)  $ by putting the left-hand side
in the form
\[%
\begin{array}
[c]{l}%
{\displaystyle\int\nolimits_{\Omega^{\pm}}}
2\mu\boldsymbol{\nabla u}\cdot\boldsymbol{\nabla v}-\mu\boldsymbol{\nabla
}\times\boldsymbol{u}\cdot\boldsymbol{\nabla}\times\boldsymbol{v}%
+\lambda\boldsymbol{\nabla}\cdot\boldsymbol{u}\ \boldsymbol{\nabla}%
\cdot\boldsymbol{v}+\boldsymbol{\Delta}^{\ast}\boldsymbol{u}\cdot
\boldsymbol{v\ }dx=%
{\displaystyle\int\nolimits_{\Omega^{\pm}}}
2\mu\left(  \boldsymbol{\nabla u}\cdot\boldsymbol{\nabla v}+\boldsymbol{\Delta
u}\cdot\boldsymbol{v}\right)  dx\\
\quad+%
{\displaystyle\int\nolimits_{\Omega^{\pm}}}
\left(  \lambda+\mu\right)  \left(  \boldsymbol{\nabla\nabla}\cdot
\boldsymbol{u}\ \cdot\boldsymbol{v}+\boldsymbol{\nabla}\cdot\boldsymbol{u}%
\ \boldsymbol{\nabla}\cdot\boldsymbol{v}\right)  dx-%
{\displaystyle\int\nolimits_{\Omega^{\pm}}}
\mu\left(  \boldsymbol{\Delta u}\cdot\boldsymbol{v}+\boldsymbol{\nabla}%
\times\boldsymbol{u}\cdot\boldsymbol{\nabla}\times\boldsymbol{v}%
+\boldsymbol{\nabla}\cdot\boldsymbol{u}\ \boldsymbol{\nabla}\cdot
\boldsymbol{v}\right)  dx.
\end{array}
\]
It is extended to $\boldsymbol{u}\in H_{\text{loc}}^{1}\left(
\boldsymbol{\Delta}^{\ast},\overline{\Omega^{\pm}}\right)  $ by usual density,
continuity and duality arguments (cf., for example, \cite{Grisvard:82} for the
case of the Laplace operator and \cite{Costabel:88-2,McLean:00} for more
general elliptic problems). Formulae (\ref{traction-M1}) and
(\ref{traction-M2}) are then a simple recast of this identity from volume
expression (\ref{M_vol}) of $\mathcal{M}^{\left(  \boldsymbol{n}\right)  }$.
\end{proof}

\begin{remark}
It is worth noting the following two important features:

\begin{enumerate}
\item A first part of the integrand in (\ref{traction-var}) is exactly the
(opposite) of the density of virtual work
\[
2\mu\boldsymbol{\nabla u}\cdot\boldsymbol{\nabla v}-\mu\boldsymbol{\nabla
}\times\boldsymbol{u}\cdot\boldsymbol{\nabla}\times\boldsymbol{v}%
+\lambda\boldsymbol{\nabla}\cdot\boldsymbol{u}\ \boldsymbol{\nabla}%
\cdot\boldsymbol{v=\Sigma u}\cdot\boldsymbol{Ev}%
\]
done by the internal stresses
\[
\boldsymbol{\Sigma u}=2\mu\boldsymbol{Eu}+\lambda\boldsymbol{\nabla}%
\cdot\boldsymbol{u}\mathbb{I}_{3}%
\]
under the virtual displacement $\boldsymbol{v}$; $\boldsymbol{Eu}=(1/2)\left(
\boldsymbol{\nabla u}+\boldsymbol{\nabla u}^{\top}\right)  $ and
$\mathbb{I}_{3}$ are the strain tensor and the $3\times3$ unit matrix respectively;

\item When $\boldsymbol{u}\in H_{\text{loc}}^{2}\left(  \overline{\Omega^{\pm
}};\mathbb{C}^{3}\right)  $, usual Green formula yields the well-known
representation formulae of the traction in terms of $\mathcal{M}^{\left(
\boldsymbol{n}\right)  }$ (cf. \cite[Formula (V, 1.16)]{Kupradze:79} and
\cite[Formula (2.2.35)]{Hsiao-Wendland:08})%
\begin{equation}
T^{\left(  \boldsymbol{n}\right)  }\boldsymbol{u}=2\mu\mathcal{M}^{\left(
\boldsymbol{n}\right)  }\boldsymbol{u}-\mu\boldsymbol{n}\times
\boldsymbol{\nabla}\times\boldsymbol{u}+\left(  \lambda+2\mu\right)
\boldsymbol{n\nabla}\cdot\boldsymbol{u} \label{Tn-M-I}%
\end{equation}%
\begin{equation}
T^{\left(  \boldsymbol{n}\right)  }\boldsymbol{u}=\mu\mathcal{M}^{\left(
\boldsymbol{n}\right)  }\boldsymbol{u}+\mu\partial_{\boldsymbol{n}%
}\boldsymbol{u}+\left(  \lambda+\mu\right)  \boldsymbol{n\nabla}%
\cdot\boldsymbol{u}. \label{Tn-M-II}%
\end{equation}

\end{enumerate}
\end{remark}

We can thus establish the representation of the traction of the single-layer
potential in terms of the traces of the Helmholtz equation potentials.

\begin{theorem}
The operators defined by $\left(  T^{\left(  \boldsymbol{n}\right)
}S\boldsymbol{p}\right)  ^{\pm}$ for $\boldsymbol{p}\in H^{-1/2}\left(
\partial\Omega;\mathbb{C}^{3}\right)  $ have the following representation%
\begin{equation}
\left(  T^{\left(  \boldsymbol{n}\right)  }S\boldsymbol{p}\right)  ^{\pm
}=\left(  \partial_{\boldsymbol{n}}V_{\kappa_{s}}\boldsymbol{p}\right)  ^{\pm
}+\boldsymbol{n\nabla}\cdot\left(  V_{\kappa_{p}}-V_{\kappa_{s}}\right)
\boldsymbol{p}-\mathcal{M}^{\left(  \boldsymbol{n}\right)  }\left(
V_{\kappa_{s}}-2\mu S\right)  \boldsymbol{p}. \label{traction-S}%
\end{equation}
In particular, the jump of the traction of the single-layer potential is given
by%
\[
\left[  T^{\left(  \boldsymbol{n}\right)  }S\boldsymbol{p}\right]
=\boldsymbol{p}.
\]
The mapping properties of these operators can be stated as follows%
\[
\left(  T^{\left(  \boldsymbol{n}\right)  }S\boldsymbol{p}\right)  ^{\pm}\in
H^{-1/2+s}\left(  \partial\Omega;\mathbb{C}^{3}\right)  \text{, }%
\]
for $\boldsymbol{p}\in H^{-1/2+s}\left(  \partial\Omega;\mathbb{C}^{3}\right)
$ and $-1/2<s<1/2.$
\end{theorem}

\begin{proof}
Keeping the general notation of Lemma\ \ref{Lemma3}, we use representation
(\ref{traction-M1}) of the traction to write%
\[
\left\langle \boldsymbol{v},\left(  T^{\left(  \boldsymbol{n}\right)
}S\boldsymbol{p}\right)  ^{\pm}\right\rangle _{1/2,\partial\Omega
}=\left\langle \boldsymbol{v},2\mu\mathcal{M}^{\left(  \boldsymbol{n}\right)
}S\boldsymbol{p}\right\rangle _{1/2,\partial\Omega}\pm%
{\displaystyle\int_{\Omega^{\pm}}}
\mu\boldsymbol{\nabla}\times S\boldsymbol{p}\cdot\boldsymbol{\nabla}%
\times\boldsymbol{v}+\left(  \lambda+2\mu\right)  \boldsymbol{\nabla}\cdot
S\boldsymbol{p}\ \boldsymbol{\nabla}\cdot\boldsymbol{v}+\boldsymbol{\Delta
}^{\ast}S\boldsymbol{p}\cdot\boldsymbol{v\ }dx.
\]
Noting that $\mu\boldsymbol{\nabla}\times S\boldsymbol{p}=\boldsymbol{\nabla
}\times V_{\kappa_{s}}\boldsymbol{p}$, $\left(  \lambda+2\mu\right)
\boldsymbol{\nabla}\cdot S\boldsymbol{p}=\boldsymbol{\nabla}\cdot
V_{\kappa_{p}}\boldsymbol{p}$, and $\boldsymbol{\Delta}^{\ast}S\boldsymbol{p}%
=-\omega^{2}\varrho S\boldsymbol{p}$ in $\Omega^{\pm}$, we get%
\[
\left\langle \boldsymbol{v},\left(  T^{\left(  \boldsymbol{n}\right)
}S\boldsymbol{p}\right)  ^{\pm}\right\rangle _{1/2,\partial\Omega
}=\left\langle \boldsymbol{v},2\mu\mathcal{M}^{\left(  \boldsymbol{n}\right)
}S\boldsymbol{p}\right\rangle _{1/2,\partial\Omega}\pm%
{\displaystyle\int_{\Omega^{\pm}}}
\boldsymbol{\nabla}\times V_{\kappa_{s}}\boldsymbol{p}\cdot\boldsymbol{\nabla
}\times\boldsymbol{v}+\boldsymbol{\nabla}\cdot V_{\kappa_{p}}\boldsymbol{p}%
\ \boldsymbol{\nabla}\cdot\boldsymbol{v}-\omega^{2}\varrho S\boldsymbol{p}%
\cdot\boldsymbol{v\ }dx.
\]
or in a more explicit form%
\[%
\begin{array}
[c]{l}%
\left\langle \boldsymbol{v},\left(  T^{\left(  \boldsymbol{n}\right)
}S\boldsymbol{p}\right)  ^{\pm}\right\rangle _{1/2,\partial\Omega
}=\left\langle \boldsymbol{v},2\mu\mathcal{M}^{\left(  \boldsymbol{n}\right)
}S\boldsymbol{p}\right\rangle _{1/2,\partial\Omega}\\
\qquad\pm%
{\displaystyle\int_{\Omega^{\pm}}}
\boldsymbol{\nabla}\times V_{\kappa_{s}}\boldsymbol{p}\cdot\boldsymbol{\nabla
}\times\boldsymbol{v}+\boldsymbol{\nabla}\cdot V_{\kappa_{p}}\boldsymbol{p}%
\ \boldsymbol{\nabla}\cdot\boldsymbol{v\ }dx\pm%
{\displaystyle\int_{\Omega^{\pm}}}
\left(  -\kappa_{s}^{2}V_{\kappa_{s}}\boldsymbol{p}-\boldsymbol{\nabla
\boldsymbol{\nabla}\cdot}V_{\kappa_{s}}\boldsymbol{\boldsymbol{p}%
}+\boldsymbol{\nabla\boldsymbol{\nabla}\cdot}V_{\kappa_{p}}%
\boldsymbol{\boldsymbol{p}}\right)  \cdot\boldsymbol{v\ }dx.
\end{array}
\]
Reorganizing the integrands, we come to%
\[%
\begin{array}
[c]{l}%
\left\langle \boldsymbol{v},\left(  T^{\left(  \boldsymbol{n}\right)
}S\boldsymbol{p}\right)  ^{\pm}\right\rangle _{1/2,\partial\Omega
}=\left\langle \boldsymbol{v},2\mu\mathcal{M}^{\left(  \boldsymbol{n}\right)
}S\boldsymbol{p}\right\rangle _{1/2,\partial\Omega}\\
\qquad\pm%
{\displaystyle\int_{\Omega^{\pm}}}
\boldsymbol{\nabla}\times V_{\kappa_{s}}\boldsymbol{p}\cdot\boldsymbol{\nabla
}\times\boldsymbol{v}+\boldsymbol{\nabla}\cdot V_{\kappa_{s}}\boldsymbol{p}%
\ \boldsymbol{\nabla}\cdot\boldsymbol{v}+\boldsymbol{\Delta}V_{\kappa_{s}%
}\boldsymbol{\boldsymbol{p}}\cdot\boldsymbol{v\ }dx\\
\qquad\pm%
{\displaystyle\int_{\Omega^{\pm}}}
\boldsymbol{\nabla\boldsymbol{\nabla}}\cdot\left(  V_{\kappa_{p}}%
-V_{\kappa_{s}}\right)  \boldsymbol{\boldsymbol{p}}\cdot\boldsymbol{v}%
+\boldsymbol{\boldsymbol{\nabla}}\cdot\left(  V_{\kappa_{p}}-V_{\kappa_{s}%
}\right)  \boldsymbol{\boldsymbol{p\nabla}}\cdot\boldsymbol{v\ }dx,
\end{array}
\]
which can also be written as
\[%
\begin{array}
[c]{l}%
\left\langle \boldsymbol{v},\left(  T^{\left(  \boldsymbol{n}\right)
}S\boldsymbol{p}\right)  ^{\pm}\right\rangle _{1/2,\partial\Omega
}=\left\langle \boldsymbol{v},2\mu\mathcal{M}^{\left(  \boldsymbol{n}\right)
}S\boldsymbol{p}\right\rangle _{1/2,\partial\Omega}\pm%
{\displaystyle\int_{\Omega^{\pm}}}
\boldsymbol{\Delta}V_{\kappa_{s}}\boldsymbol{\boldsymbol{p}}\cdot
\boldsymbol{v}+\boldsymbol{\nabla}V_{\kappa_{s}}\boldsymbol{p}\cdot
\boldsymbol{\nabla v}\ dx\\
\qquad\pm%
{\displaystyle\int_{\Omega^{\pm}}}
\boldsymbol{\nabla}\times V_{\kappa_{s}}\boldsymbol{p}\cdot\boldsymbol{\nabla
}\times\boldsymbol{v}+\boldsymbol{\nabla}\cdot V_{\kappa_{s}}\boldsymbol{p}%
\ \boldsymbol{\nabla}\cdot\boldsymbol{v}-\boldsymbol{\nabla}V_{\kappa_{s}%
}\boldsymbol{p}\cdot\boldsymbol{\nabla v\ }dx\\
\qquad\pm%
{\displaystyle\int_{\Omega^{\pm}}}
\boldsymbol{\nabla\boldsymbol{\nabla}}\cdot\left(  V_{\kappa_{p}}%
-V_{\kappa_{s}}\right)  \boldsymbol{\boldsymbol{p}}\cdot\boldsymbol{v}%
+\boldsymbol{\boldsymbol{\nabla}}\cdot\left(  V_{\kappa_{p}}-V_{\kappa_{s}%
}\right)  \boldsymbol{\boldsymbol{p\nabla}}\cdot\boldsymbol{v\ }dx.
\end{array}
\]
Volume expression (\ref{M_vol})\ of $\mathcal{M}^{\left(  \boldsymbol{n}%
\right)  }$ and usual Green formula directly yield (\ref{traction-S}). The
jump of $T^{\left(  \boldsymbol{n}\right)  }S\boldsymbol{p}$ directly follows
from that of the normal derivative of the single-layer potential of the
Helmholtz equation. The mapping properties are obtained in the same way than
those related to the traces of the double-layer potential.
\end{proof}

\begin{remark}
Representation formula (\ref{traction-S}) establishes the duality identity
\[
\left\langle \boldsymbol{\psi},\left(  T^{\left(  \boldsymbol{n}\right)
}S\boldsymbol{p}\right)  ^{\pm}\right\rangle _{1/2,\partial\Omega
}=-\left\langle \left(  K\boldsymbol{\psi}\right)  ^{\mp},\boldsymbol{p}%
\right\rangle _{1/2,\partial\Omega},\
\]
for $\boldsymbol{\psi}\in H^{1/2}\left(  \partial\Omega;\mathbb{C}^{3}\right)
$ and $\boldsymbol{p}\in H^{-1/2}\left(  \partial\Omega;\mathbb{C}^{3}\right)
,$ from the corresponding formula for the potentials of the Helmholtz equation
without resorting to the general theory for elliptic systems \cite[p.
211]{McLean:00}.
\end{remark}

Now we address the perhaps most important issue in this paper: a suitable
regularization of the hypersingular kernels arising in the representation of
the traction of the double-layer potential. As said above, we here extend two
regularizations, devised by Le~Lou\"{e}r \cite{LeLouer:14,Lelouerv1:12} for a
geometry of class $\mathcal{C}^{2}$, to a Lipschitz domain.

The first regularization is based on formula (\ref{EW-K-II}), and can be
viewed, at some extent, as a generalization of the static elasticity case
derived by Han (cf. \cite{Han:94} and \cite[Lemma 2.3.3]{Hsiao-Wendland:08}).

\begin{theorem}
For $\boldsymbol{\psi}\in H^{1/2}\left(  \partial\Omega;\mathbb{C}^{3}\right)
$, the traction of the double-layer potential on each side of $\partial\Omega$
is given by
\begin{align}
\left(  T^{\left(  \boldsymbol{n}\right)  }K\boldsymbol{\psi}\right)  ^{\pm}
&  =\mu\left(  \left(  \partial_{\boldsymbol{n}}N_{\kappa_{s}}\boldsymbol{\psi
}\right)  +\mathcal{M}^{\left(  \boldsymbol{n}\right)  }\left(  N_{\kappa_{s}%
}\boldsymbol{\psi}\right)  ^{\pm}-\left(  \partial_{\boldsymbol{n}}%
V_{\kappa_{s}}\mathcal{M}^{\left(  \boldsymbol{n}\right)  }\boldsymbol{\psi
}\right)  ^{\pm}\right)  \nonumber\\
&  +2\mu\left(  \mathcal{M}^{\left(  \boldsymbol{n}\right)  }%
\boldsymbol{\nabla}\left(  V_{\kappa_{p}}-V_{\kappa_{s}}\right)
\boldsymbol{n}\cdot\boldsymbol{\psi}-\boldsymbol{n\nabla}\cdot\left(
V_{\kappa_{p}}-V_{\kappa_{s}}\right)  \mathcal{M}^{\left(  \boldsymbol{n}%
\right)  }\boldsymbol{\psi}\right)  \nonumber\\
&  +\mathcal{M}^{\left(  \boldsymbol{n}\right)  }\left(  3\mu V_{\kappa_{s}%
}-4\mu^{2}S\right)  \mathcal{M}^{\left(  \boldsymbol{n}\right)  }%
\boldsymbol{\psi}-\omega^{2}\varrho\boldsymbol{n}\left(  V_{\kappa_{p}%
}-V_{\kappa_{s}}\right)  \boldsymbol{n}\cdot\boldsymbol{\psi}.\label{T_Alter}%
\end{align}
In particular, $\left[  T^{\left(  \boldsymbol{n}\right)  }K\boldsymbol{\psi
}\right]  =0$ and $\left(  T^{\left(  \boldsymbol{n}\right)  }%
K\boldsymbol{\psi}\right)  ^{\pm}=$ $T^{\left(  \boldsymbol{n}\right)
}K\boldsymbol{\psi}$ defines a bounded operator from $H^{1/2+s}(\allowbreak
\partial\Omega;\mathbb{C}^{3})$ into $H^{-1/2+s}\left(  \partial
\Omega;\mathbb{C}^{3}\right)  $ for $-1/2<s<1/2$.
\end{theorem}

\begin{proof}
The calculations follow those in \cite[Lemma 2.3]{Lelouerv1:12}. They are
however carried out here on the potentials instead on the kernels. The
approach in \cite[Lemma 2.3]{Lelouerv1:12}, more or less explicitly, requires
a smooth extension of the unit normal in a neighborhood of $\partial\Omega$,
which, of course, is not available for a Lipschitz geometry. The derivation is
based on the decomposition of the double-layer potential in three terms%
\[
K\boldsymbol{\psi}=\underset{\boldsymbol{w}_{0}}{\underbrace{N_{\kappa_{s}%
}\boldsymbol{\psi}+V_{\kappa_{s}}\mathcal{M}^{\left(  \boldsymbol{n}\right)
}\boldsymbol{\psi}}}+\underset{\boldsymbol{w}_{1}%
}{\underbrace{\boldsymbol{\nabla}\left(  V_{\kappa_{p}}-V_{\kappa_{s}}\right)
\boldsymbol{n}\cdot\boldsymbol{\psi}}}-2\mu S\mathcal{M}^{\left(
\boldsymbol{n}\right)  }\boldsymbol{\psi}.
\]
The last term is just a multiple of the single-layer potential created by the
density $\mathcal{M}^{\left(  \boldsymbol{n}\right)  }\boldsymbol{\psi}$: thus
the corresponding tractions are given by (\ref{traction-S})%
\begin{align*}
\left(  T^{\left(  \boldsymbol{n}\right)  }\left(  -2\mu S\mathcal{M}^{\left(
\boldsymbol{n}\right)  }\boldsymbol{\psi}\right)  \right)  ^{\pm}  &
=-2\mu\left(  \partial_{\boldsymbol{n}}V_{\kappa_{s}}\mathcal{M}^{\left(
\boldsymbol{n}\right)  }\boldsymbol{\psi}\right)  ^{\pm}\\
& -2\mu\left(  \boldsymbol{n\nabla}\cdot\left(  V_{\kappa_{p}}-V_{\kappa_{s}%
}\right)  \mathcal{M}^{\left(  \boldsymbol{n}\right)  }\boldsymbol{\psi
}-\mathcal{M}^{\left(  \boldsymbol{n}\right)  }\left(  V_{\kappa_{p}%
}-V_{\kappa_{s}}\right)  \mathcal{M}^{\left(  \boldsymbol{n}\right)
}\boldsymbol{\psi}\right)  .
\end{align*}
The second term is in $H_{\text{loc}}^{2}\left(  \overline{\Omega^{\pm}%
};\mathbb{C}^{3}\right)  $. The corresponding traction can be calculated using
the direct definition and the fact that $\boldsymbol{\nabla}\times
\boldsymbol{w}_{1}=0$ and $\boldsymbol{\nabla}\cdot\boldsymbol{w}_{1}=\left(
-\kappa_{p}^{2}V_{\kappa_{p}}+\kappa_{s}^{2}V_{\kappa_{s}}\right)
\boldsymbol{n}\cdot\boldsymbol{\psi}$
\[
T^{\left(  \boldsymbol{n}\right)  }\boldsymbol{w}_{1}=2\mu\mathcal{M}^{\left(
\boldsymbol{n}\right)  }\boldsymbol{w}_{1}+\left(  \lambda+2\mu\right)
\left(  -\kappa_{p}^{2}V_{\kappa_{p}}+\kappa_{s}^{2}V_{\kappa_{s}}\right)
\boldsymbol{n}\cdot\boldsymbol{\psi}.
\]
For the last term, we first observe that $\boldsymbol{w}_{0}\in H_{\text{loc}%
}^{1}\left(  \overline{\Omega^{\pm}};\mathbb{C}^{3}\right)  $ and
$\boldsymbol{\Delta w}_{0}=-\kappa_{s}^{2}\boldsymbol{w}_{0}$ in $\Omega^{\pm
}$ since $\boldsymbol{w}_{0}$ is a combination of layer potentials of the
Helmholtz equation corresponding to the wavenumber $\kappa_{s}$ with
respective densities $\boldsymbol{\psi}\in H^{1/2}\left(  \partial
\Omega;\mathbb{C}^{3}\right)  $ and $\mathcal{M}^{\left(  \boldsymbol{n}%
\right)  }\boldsymbol{\psi}\in H^{-1/2}\left(  \partial\Omega;\mathbb{C}%
^{3}\right)  $. Next using (\ref{EW-K-half}), we can write%
\[
\boldsymbol{\nabla}\cdot\boldsymbol{w}_{0}=\boldsymbol{\Delta}V_{\kappa_{s}%
}\boldsymbol{n}\cdot\boldsymbol{\psi}=-\kappa_{s}^{2}V_{\kappa_{s}%
}\boldsymbol{n}\cdot\boldsymbol{\psi}\in H_{\text{loc}}^{1}\left(
\overline{\Omega^{\pm}}\right)
\]
so that by Green's formula we readily get that $\left(  T^{\left(
\boldsymbol{n}\right)  }\boldsymbol{w}_{0}\right)  ^{\pm}$ can be expressed by
(\ref{Tn-M-II}) from (\ref{traction-M2}) so arriving to%
\begin{multline*}
\left(  T^{\left(  \boldsymbol{n}\right)  }\boldsymbol{w}_{0}\right)  ^{\pm
}=\mu\mathcal{M}^{\left(  \boldsymbol{n}\right)  }\boldsymbol{w}_{0}%
+\mu\left(  \partial_{\boldsymbol{n}}N_{\kappa_{s}}\boldsymbol{\psi}\right)
|_{\partial\Omega}+\mu\left(  \partial_{\boldsymbol{n}}V_{\kappa_{s}%
}\mathcal{M}^{\left(  \boldsymbol{n}\right)  }\boldsymbol{\psi}\right)  ^{\pm
}\\
-\kappa_{s}^{2}\left(  \lambda+\mu\right)  \boldsymbol{n}V_{\kappa_{s}%
}\boldsymbol{n}\cdot\boldsymbol{\psi}.
\end{multline*}
It is enough to collect the above three terms to obtain (\ref{T_Alter}). The
rest of the proof is obtained from the jump and mapping properties of the
layer potentials of the Helmholtz equation (cf. \cite{Costabel:88-2} or
\cite[p. 202]{McLean:00}).
\end{proof}

\begin{remark}
Actually, representation formula (\ref{T_Alter}) leads to an expression of
$T^{\left(  \boldsymbol{n}\right)  }K\boldsymbol{\psi}$ where the integrals
are converging in the usual meaning, in other words with no need for Cauchy
principal values or Hadamard finite parts to be defined. This property is
provided by the fact that the term $\partial_{\boldsymbol{n}}N_{\kappa_{s}%
}\boldsymbol{\psi}$ can be represented in a variational form using Hamdi's
regularization formula \cite{Hamdi:81}%
\[
\left\langle \boldsymbol{\varphi},\partial_{\boldsymbol{n}}N_{\kappa_{s}%
}\boldsymbol{\psi}\right\rangle _{1/2,\partial\Omega}=\sum_{j=1}%
^{3}\left\langle \boldsymbol{n}\times\boldsymbol{\nabla}\varphi_{j}%
,V_{\kappa_{s}}\boldsymbol{n}\times\boldsymbol{\nabla}\psi_{j}\right\rangle
_{1/2,\partial\Omega}-\int_{\partial\Omega}\varphi_{j}\boldsymbol{n}\cdot
V_{\kappa_{s}}\left(  \psi_{j}\boldsymbol{n}\right)  ds
\]
with $\boldsymbol{\varphi}\in H^{1/2}\left(  \partial\Omega;\mathbb{C}%
^{3}\right)  $, $\psi_{j}$ and $\varphi_{j}$ being the components of
$\boldsymbol{\psi}$ and $\boldsymbol{\varphi}$ respectively (see \cite[p.
289]{McLean:00} for a comprehensive proof).
\end{remark}

\begin{remark}
Le\ Lou\"{e}r \cite[Lemma 2.3]{LeLouer:14} gave a second representation
formula for $T^{\left(  \boldsymbol{n}\right)  }K\boldsymbol{\psi}$%
\begin{equation}%
\begin{array}
[c]{l}%
T^{\left(  \boldsymbol{n}\right)  }K\boldsymbol{\psi}=\mu\boldsymbol{\nabla
}_{\partial\Omega}(V_{\kappa_{s}}\boldsymbol{\nabla}_{\partial\Omega}%
\cdot\boldsymbol{\psi}\times\boldsymbol{n})\times\boldsymbol{n}\\
\qquad+2\mu\left(  \mathcal{M}^{\left(  \boldsymbol{n}\right)  }\left(
N_{\kappa_{s}}\boldsymbol{\psi}\right)  ^{\pm}-\left(  \partial
_{\boldsymbol{n}}V_{\kappa_{s}}\mathcal{M}^{\left(  \boldsymbol{n}\right)
}\boldsymbol{\psi}\right)  ^{\pm}\right) \\
\qquad+2\mu\left(  \mathcal{M}^{\left(  \boldsymbol{n}\right)  }%
\boldsymbol{\nabla}\left(  V_{\kappa_{p}}-V_{\kappa_{s}}\right)
\boldsymbol{n}\cdot\boldsymbol{\psi}-\boldsymbol{n\nabla}\cdot\left(
V_{\kappa_{p}}-V_{\kappa_{s}}\right)  \mathcal{M}^{\left(  \boldsymbol{n}%
\right)  }\boldsymbol{\psi}\right) \\
\qquad+(4/\kappa_{s}^{2})\mathcal{M}^{\left(  \boldsymbol{n}\right)
}\boldsymbol{\nabla\nabla\cdot}\left(  V_{\kappa_{p}}-V_{\kappa_{s}}\right)
\mathcal{M}^{\left(  \boldsymbol{n}\right)  }\boldsymbol{\psi}\\
\qquad-\omega^{2}\varrho\left(  \boldsymbol{n}\times V_{\kappa_{s}}\left(
\boldsymbol{\psi}\times\boldsymbol{n}\right)  +\boldsymbol{n}V_{\kappa_{p}%
}\boldsymbol{n\cdot\psi}\right)  .
\end{array}
\label{TN_K2}%
\end{equation}
The derivation of this author can similarly be adapted to deal with a
Lipschitz geometry starting this once from representation formula
(\ref{EW-K-I}) and using variational form (\ref{traction-M1}) for the
traction. The mapping properties of the related operator result as above from
those of the layer potentials of the Helmholtz equation and of those of the
tangential vector rotational and the surface rotational given in
Corollary~\ref{tang_vect_rot}.
\end{remark}

\section{The two-dimensional case\label{section4}}

We limit ourselves here to the case where both the geometry and the mechanical
characteristics of the elastic medium are invariant to translations along the
$x_{3}$-axis. We first examine what happens to the G\"{u}nter derivatives when
applied to a function independent of the variable $x_{3}$. We next use the
relation linking the 2D and 3D Green kernels of the Helmholtz equation to
express the two-dimensional elastic wave potentials similarly as above in
$\mathbb{R}^{3}$.

\subsection{Two-dimensional G\"{u}nter derivatives}

In this part, we assume that the geometry is described as follows:
$\Omega^{\pm}=\Omega_{\bot}^{\pm}\times\left(  -\infty,+\infty\right)  $ where
$\Omega_{\bot}^{+}$ is a bounded 2D Lipschitz domain of the plane and
$\Omega_{\bot}^{-}=\mathbb{R}^{2}\setminus\overline{\Omega_{\bot}^{+}}$ is its
complement. Any vector field $\boldsymbol{u}$, depending only on the
transverse variable $(x_{1},x_{2})$, can be written as the superposition of a
plane vector field $\boldsymbol{u}_{\bot}$ and a scalar field $u_{3}$,
respectively called the plane and the anti-plane components of $\boldsymbol{u}%
$, according to the decomposition%
\[
\boldsymbol{u}\left(  x_{1},x_{2}\right)  =\boldsymbol{u}_{\bot}\left(
x_{1},x_{2}\right)  +u_{3}\left(  x_{1},x_{2}\right)  \boldsymbol{e}_{3}.
\]
Recall that $\left\{  \boldsymbol{e}_{j}\right\}  _{j=1}^{3}$ denotes the
canonical basis of the space. The unit normal $\boldsymbol{n}$ to
$\partial\Omega$ is independent of $x_{3}$, and verifies $n_{3}=0$. As a
result, we do not distinguish between $\boldsymbol{n}$ and its plane component
$\boldsymbol{n}_{\bot}$. Subscript $\bot$ is used to denote 2D analogs of 3D
symbols. Let us just mention that $\boldsymbol{\nabla}_{\bot}\times
\boldsymbol{u}_{\bot}$ and $\boldsymbol{\nabla}_{\bot}\times u_{3}$ are the
scalar curl and the vector curl and are defined by%
\[
\boldsymbol{\nabla}_{\bot}\times\boldsymbol{u}_{\bot}=\partial_{x_{1}}%
u_{2}-\partial_{x_{2}}u_{1}\text{, }\boldsymbol{\nabla}_{\bot}\times
u_{3}=\partial_{x_{2}}u_{3}\boldsymbol{e}_{1}-\partial_{x_{1}}u_{3}%
\boldsymbol{e}_{2}.
\]

Let $u$ be a function independent of $x_{3}$. We readily get that%
\[
\mathcal{M}_{i3}^{\left(  \boldsymbol{n}\right)  }u=n_{3}\partial_{x_{i}%
}u-n_{i}\partial_{x_{3}}u=0.
\]
As a result, only two G\"{u}nter derivatives are not zero
\[
\mathcal{M}_{21}^{\left(  \boldsymbol{n}\right)  }u=-\mathcal{M}_{12}^{\left(
\boldsymbol{n}\right)  }u=n_{1}\partial_{x_{2}}u-n_{2}\partial_{x_{1}%
}u=\partial_{\boldsymbol{\tau}}u
\]
with $\boldsymbol{\tau}=R_{\pi/2}\boldsymbol{n}$, $R_{\theta}$ being the
counterclockwise rotation around the $x_{3}$-axis by $\theta$. In other
words,
\[
\mathcal{M}_{21}^{\left(  \boldsymbol{n}\right)  }u=-\mathcal{M}_{12}^{\left(
\boldsymbol{n}\right)  }u=\partial_{s}u
\]
where $s$ is the curvilinear abscissa of $\partial\Omega_{\bot}$ growing in
the counterclockwise direction. The following version of
Theorem~\ref{Gunter_map} is more usual.

\begin{theorem}
Under the above general assumptions, operator $\partial_{s}$ is bounded from
$H^{s}\left(  \partial\Omega_{\bot}\right)  $ into $H^{s-1}\left(
\partial\Omega_{\bot}\right)  $ for $0\leq s\leq1$.
\end{theorem}

\begin{remark}
G\"{u}nter derivative matrix $\mathcal{M}^{\left(  \boldsymbol{n}\right)  }$
reduces to an operator of a particularly simple form
\[
\mathcal{M}^{\left(  \boldsymbol{n}\right)  }\boldsymbol{u}=\mathcal{M}_{\bot
}^{\left(  \boldsymbol{n}\right)  }\boldsymbol{u}_{\bot}=R_{\pi/2}\partial
_{s}\boldsymbol{u}_{\bot}=\boldsymbol{e}_{3}\times\partial_{s}\boldsymbol{u}.
\]

\end{remark}

\subsection{Two-dimensional elastic waves layer potentials}

Noting that
\[
\mu\boldsymbol{\Delta u}+\left(  \mu+\lambda\right)  \boldsymbol{\nabla\nabla
}\cdot\boldsymbol{u}=\left[
\begin{array}
[c]{c}%
\mu\boldsymbol{\Delta}_{\bot}\boldsymbol{u}_{\bot}+\left(  \mu+\lambda\right)
\boldsymbol{\nabla}_{\bot}\boldsymbol{\nabla}_{\bot}\cdot\boldsymbol{u}_{\bot
}\\
\mu\Delta_{\bot}u_{3}%
\end{array}
\right]  ,
\]
we readily get that the plane $\boldsymbol{u}_{\bot}$ and the antiplane
$u_{3}$ components of $\boldsymbol{u}$ are uncoupled at the level of the
propagation equations.

Finally, the plane component $\left(  T^{\left(  \boldsymbol{n}\right)
}\boldsymbol{u}\right)  _{\bot}$and the antiplane $\left(  T^{\left(
\boldsymbol{n}\right)  }\boldsymbol{u}\right)  _{3}$ one of the traction,
corresponding to a field $\boldsymbol{u}$ independent of $x_{3}$, respectively
depend on the plane displacement $\boldsymbol{u}_{\bot}$ and the antiplane one
$u_{3}$ only%
\begin{align*}
\left(  T^{\left(  \boldsymbol{n}\right)  }\boldsymbol{u}\right)  _{\bot}  &
=T_{\bot}^{\left(  \boldsymbol{n}\right)  }\boldsymbol{u}_{\bot}=2\mu
\partial_{\boldsymbol{n}}\boldsymbol{u}_{\bot}+\boldsymbol{n\nabla}_{\bot
}\cdot\boldsymbol{u}_{\bot}+\boldsymbol{\tau\nabla}_{\bot}\times
\boldsymbol{u}_{\bot}\\
\left(  T^{\left(  \boldsymbol{n}\right)  }\boldsymbol{u}\right)  _{3}  &
=T_{3}^{\left(  \boldsymbol{n}\right)  }u_{3}=\mu\partial_{\boldsymbol{n}%
}u_{3},
\end{align*}

The expressions of the layer potentials and the related boundary integral
operators can thus be obtained in a simple way using the following integral
representation of the 2D fundamental solution of the Helmholtz equation%
\[
\frac{i}{4}H_{0}^{\left(  1\right)  }\left(  \kappa r\right)  =\int_{-\infty
}^{+\infty}\frac{\exp\left(  i\kappa\sqrt{r^{2}+x_{3}^{2}}\right)  }{4\pi
\sqrt{r^{2}+x_{3}^{2}}}dx_{3}\text{ for }r>0\text{,}%
\]
which is classically obtained by the change of variable $x_{3}=\sinh t$ from
the Mehler-Sonine integrals \cite[Formulae 10.9.9]{NIST:10}
\[
\frac{i}{4}H_{0}^{\left(  1\right)  }\left(  \kappa r\right)  =\frac{1}{4\pi
}\int_{-\infty}^{+\infty}\exp\left(  i\kappa r\cosh t\right)  dt.
\]
For simplicity, we avoid to distinguish by subscript $\bot$ the single-layer
and the double-layer potentials related to the Helmholtz equation in 2D%
\[
V_{\kappa}p(x_{1},x_{2})=\int_{\partial\Omega_{\bot}}\frac{i}{4}H_{0}^{\left(
1\right)  }\left(  \kappa\sqrt{\left(  x_{1}-y_{1}\right)  ^{2}+\left(
x_{2}-y_{2}\right)  ^{2}}\right)  p(y_{1},y_{2})ds_{\left(  y_{1}%
,y_{2}\right)  },
\]%
\[
N_{\kappa}\varphi(x_{1},x_{2})=-\int_{\partial\Omega_{\bot}}\frac{i}%
{4}\partial_{\boldsymbol{n}_{y_{1},y_{2}}}H_{0}^{\left(  1\right)  }\left(
\kappa\sqrt{\left(  x_{1}-y_{1}\right)  ^{2}+\left(  x_{2}-y_{2}\right)  ^{2}%
}\right)  \varphi(y_{1},y_{2})ds_{\left(  y_{1},y_{2}\right)  },
\]
leaving the context to define whether it is the 2D case or the 3D one which is considered.

Each potential or boundary integral operator related to two-dimensional
elastic waves is decomposed in its plane and antiplane parts:

\begin{itemize}
\item Single-layer potential \newline$\hspace{0.2cm}%
\begin{array}
[c]{l}%
S\boldsymbol{p}=S_{\bot}\boldsymbol{p}_{\bot}+\left(  S_{3}p_{3}\right)
\boldsymbol{e}_{3},\\
S_{\bot}\boldsymbol{p}_{\bot}=\frac{1}{\omega^{2}\varrho}\left(  \kappa
_{s}^{2}V_{\kappa_{s}}\boldsymbol{p}_{\bot}+\boldsymbol{\nabla}_{\bot
}\boldsymbol{\nabla}_{\bot}\cdot\left(  V_{\kappa_{s}}-V_{\kappa_{p}}\right)
\boldsymbol{p}_{\bot}\right)  ,\\
S_{3}p_{3}=\frac{1}{\mu}V_{\kappa_{s}}p_{3},
\end{array}
$

\item Double-layer potential\newline$\hspace{0.2cm}%
\begin{array}
[c]{l}%
K\boldsymbol{\psi}=K_{\bot}\boldsymbol{\psi}_{\bot}+\left(  K_{3}\psi
_{3}\right)  \boldsymbol{e}_{3},\\
K_{\bot}\boldsymbol{\psi}_{\bot}=N_{\kappa_{s}}\boldsymbol{\psi}_{\bot
}+\left(  V_{\kappa_{s}}-2\mu S_{\bot}\right)  \mathcal{M}_{\bot}^{\left(
\boldsymbol{n}\right)  }\boldsymbol{\psi}_{\bot}+\text{ }\boldsymbol{\nabla
}_{\bot}\left(  V_{\kappa_{p}}-V_{\kappa_{s}}\right)  \boldsymbol{n}%
\cdot\boldsymbol{\psi}_{\bot},\\
K_{3}\psi_{3}=N_{\kappa_{s}}\psi_{3},
\end{array}
$

\item Traction of the single-layer potential\newline$\hspace{0.2cm}%
\begin{array}
[c]{l}%
T^{\left(  \boldsymbol{n}\right)  }S\boldsymbol{p}=T_{\bot}^{\left(
\boldsymbol{n}\right)  }S_{\bot}\boldsymbol{p}_{\bot}+\left(  T_{3}^{\left(
\boldsymbol{n}\right)  }S_{3}p_{3}\right)  \boldsymbol{e}_{3},\\
T_{\bot}^{\left(  \boldsymbol{n}\right)  }S_{\bot}\boldsymbol{p}_{\bot
}=\left(  \partial_{\boldsymbol{n}}V_{\kappa_{s}}\boldsymbol{p}\right)  ^{\pm
}+\boldsymbol{n\nabla}_{\bot}\cdot\left(  V_{\kappa_{p}}-V_{\kappa_{s}%
}\right)  \boldsymbol{p}_{\bot}\\
\quad-\mathcal{M}_{\bot}^{\left(  \boldsymbol{n}\right)  }\left(
V_{\kappa_{s}}-2\mu S_{\bot}\right)  \boldsymbol{p}_{\bot},\\
T_{3}^{\left(  \boldsymbol{n}\right)  }S_{3}p_{3}=\left(  \partial
_{\boldsymbol{n}}V_{\kappa_{s}}p_{3}\right)  ^{\pm}.
\end{array}
$

\item Traction of the double-layer potential\newline$\hspace{0.2cm}%
\begin{array}
[c]{l}%
T^{\left(  \boldsymbol{n}\right)  }K\boldsymbol{\psi=}T_{\bot}^{\left(
\boldsymbol{n}\right)  }K_{\bot}\boldsymbol{\psi}_{\bot}+\left(
T_{3}^{\left(  \boldsymbol{n}\right)  }K_{3}\psi_{3}\right)  \boldsymbol{e}%
_{3},\\
T_{\bot}^{\left(  \boldsymbol{n}\right)  }K_{\bot}\boldsymbol{\psi}_{\bot}%
=\mu\left(  \partial_{\boldsymbol{n}}N_{\kappa_{s}}\boldsymbol{\psi}_{\bot
}+\mathcal{M}_{\bot}^{\left(  \boldsymbol{n}\right)  }\left(  N_{\kappa_{s}%
}\boldsymbol{\psi}_{\bot}\right)  ^{\pm}-\left(  \partial_{\boldsymbol{n}%
}V_{\kappa_{s}}\mathcal{M}_{\bot}^{\left(  \boldsymbol{n}\right)
}\boldsymbol{\psi}_{\bot}\right)  ^{\pm}\right) \\
\quad+2\mu\left(  \mathcal{M}_{\bot}^{\left(  \boldsymbol{n}\right)
}\boldsymbol{\nabla}_{\bot}\left(  V_{\kappa_{p}}-V_{\kappa_{s}}\right)
\boldsymbol{n}\cdot\boldsymbol{\psi}_{\bot}-\boldsymbol{n\nabla}_{\bot}%
\cdot\left(  V_{\kappa_{p}}-V_{\kappa_{s}}\right)  \mathcal{M}_{\bot}^{\left(
\boldsymbol{n}\right)  }\boldsymbol{\psi}_{\bot}\right) \\
\quad+\left(  \mathcal{M}_{\bot}^{\left(  \boldsymbol{n}\right)  }\left(  3\mu
V_{\kappa_{s}}-4\mu^{2}S_{\bot}\right)  \mathcal{M}_{\bot}^{\left(
\boldsymbol{n}\right)  }\boldsymbol{\psi}-\omega^{2}\varrho\boldsymbol{n}%
\left(  V_{\kappa_{p}}-V_{\kappa_{s}}\right)  \boldsymbol{n}\cdot
\boldsymbol{\psi}_{\bot}\right)  ,\\
T_{3}^{\left(  \boldsymbol{n}\right)  }K_{3}\psi_{3}=\mu\partial
_{\boldsymbol{n}}N_{\kappa_{s}}\psi_{3}.
\end{array}
$
\end{itemize}

\begin{remark}
Another expression for the traction of the double-layer potential
\[%
\begin{array}
[c]{l}%
T^{\left(  \boldsymbol{n}\right)  }K\boldsymbol{\psi=}T_{\bot}^{\left(
\boldsymbol{n}\right)  }K_{\bot}\boldsymbol{\psi}_{\bot}+\left(
T_{3}^{\left(  \boldsymbol{n}\right)  }K_{3}\psi_{3}\right)  \boldsymbol{e}%
_{3},\\
T_{\bot}^{\left(  \boldsymbol{n}\right)  }K_{\bot}\boldsymbol{\psi}_{\bot
}=2\mu\left(  \mathcal{M}_{\bot}^{\left(  \boldsymbol{n}\right)  }\left(
N_{\kappa_{s}}\boldsymbol{\psi}_{\bot}\right)  ^{\pm}-\left(  \partial
_{\boldsymbol{n}}V_{\kappa_{s}}\mathcal{M}_{\bot}^{\left(  \boldsymbol{n}%
\right)  }\boldsymbol{\psi}_{\bot}\right)  ^{\pm}\right) \\
\qquad\qquad+2\mu\left(  \mathcal{M}_{\bot}^{\left(  \boldsymbol{n}\right)
}\boldsymbol{\nabla}_{\bot}\left(  V_{\kappa_{p}}-V_{\kappa_{s}}\right)
\boldsymbol{n}\cdot\boldsymbol{\psi}_{\bot}-\boldsymbol{n\nabla}_{\bot}%
\cdot\left(  V_{\kappa_{p}}-V_{\kappa_{s}}\right)  \mathcal{M}_{\bot}^{\left(
\boldsymbol{n}\right)  }\boldsymbol{\psi}_{\bot}\right) \\
\qquad\qquad+(4/\kappa_{s}^{2})\mathcal{M}_{\bot}^{\left(  \boldsymbol{n}%
\right)  }\boldsymbol{\nabla}_{\bot}\boldsymbol{\nabla}_{\bot}%
\boldsymbol{\cdot}\left(  V_{\kappa_{p}}-V_{\kappa_{s}}\right)  \mathcal{M}%
_{\bot}^{\left(  \boldsymbol{n}\right)  }\boldsymbol{\psi}_{\bot}\\
\qquad\qquad-\omega^{2}\varrho\left(  \boldsymbol{\tau}V_{\kappa_{s}}\left(
\boldsymbol{\psi}_{\bot}\cdot\boldsymbol{\tau}\right)  +\boldsymbol{n}%
V_{\kappa_{p}}\boldsymbol{n\cdot\psi}_{\bot}\right)  ,\\
T_{3}^{\left(  \boldsymbol{n}\right)  }K_{3}\psi_{3}=-\mu\partial_{s}%
V_{\kappa_{s}}\partial_{s}\psi_{3}-\omega^{2}\varrho\boldsymbol{\tau\cdot
}V_{\kappa_{s}}\left(  \psi_{3}\boldsymbol{\tau}\right)  .
\end{array}
\]
can also be obtained from (\ref{TN_K2}).
\end{remark}

\bibliographystyle{elsarticle-num}
\bibliography{acompat,bendali}

\end{document}